\documentclass[twoside,12pt]{article}
\usepackage{amsmath}
\usepackage{amsthm}
\usepackage{amssymb}
\usepackage{hyperref}
\usepackage{cite}
\usepackage[usenames,dvipsnames]{color}
\usepackage{xcolor}
\definecolor{rosso}{rgb}{0.8,0,0}
\def\pier #1{{\color{blue}#1}}
\def\takeshi #1{{\color{magenta}#1}}

\def\fukao #1{{\color{magenta}#1}}
\def\pier #1{#1}
\def\takeshi #1{#1}

\def\fukao #1{#1}

\def\enne{\mathbb{N}}
\def\zeta{\mathbb{Z}}

\def\erre{\mathbb{R}}

\def\eps{\varepsilon}

\def\bH{\mbox{\boldmath $H$}}
\def\bV{\mbox{\boldmath $V$}}
\def\bZ{\mbox{\boldmath $Z$}}
\def\bW{\mbox{\boldmath $W$}}
\def\bsH{\mbox{\scriptsize\boldmath $H$}}
\def\bsV{\mbox{\scriptsize\boldmath $V$}}
\def\bsZ{\mbox{\scriptsize\boldmath $Z$}}
\def\bsW{\mbox{\scriptsize\boldmath $W$}}
\def\b#1{\mbox{\boldmath $#1$}}

\def\beq{\begin{equation}}
\def\eeq{\end{equation}}

\def\to{\rightarrow}
\def\wto{\rightharpoonup}
\def\wstarto{\stackrel{*}{\rightharpoonup}}
\def\embed{\hookrightarrow}

\def\norm #1{\left\|#1\right\|}

\def\sp #1#2{\left<#1,#2\right>}
\newcommand\ip\sp

\newtheorem{thm}{Theorem}[section]

\newtheorem{lem}[thm]{Lemma}
\newtheorem{prop}[thm]{Proposition}
\newtheorem{defin}[thm]{Definition}
\newtheorem{remark}[thm]{Remark}

\title{The Cahn--Hilliard equation with forward-backward 
dynamic boundary condition
via vanishing viscosity}

\author{Pierluigi Colli\\
Dipartimento di Matematica, Universit\`a degli Studi di Pavia\\
Via Ferrata~1, 27100 Pavia, Italy\\
E-mail: \texttt{pierluigi.colli@unipv.it}\\
\and \\ Takeshi Fukao\\
Department of Mathematics, Faculty of Education\\
Kyoto University of Education\\
1~Fujinomori, Fukakusa, Fushimi-ku, Kyoto~612-8522 Japan\\
E-mail: \texttt{fukao@kyokyo-u.ac.jp}
\and \\ Luca Scarpa\\
Department of Mathematics, Politecnico di Milano\\
Via E.~Bonardi 9, 20133 Milano, Italy\\
E-mail: \texttt{luca.scarpa@polimi.it}}
 
\date{}

\newcommand\testopari{\sc Pierluigi Colli, Takeshi Fukao, and Luca Scarpa}
\newcommand\testodispari{\sc {C}ahn--{H}illiard with forward-backward 
dynamic boundary condition}
\begin{document}

\maketitle

\begin{abstract}
An asymptotic analysis for a system with equation and dynamic boundary condition of Cahn\takeshi{--}Hilliard \takeshi{type} is carried out as the coefficient of the surface diffusion acting on the phase variable 
tends to $0$, thus obtaining a forward-backward dynamic boundary condition at the limit. 
This is done in a very general setting, with nonlinear terms admitting maximal monotone graphs 
both in the bulk and on the boundary. The two graphs are related by a growth condition, 
with the boundary graph that dominates the other one. It turns out that 
in the limiting procedure the solution of the problem looses some regularity
and the limit equation has to be properly interpreted in the sense of a subdifferential 
inclusion. However, the limit problem is still
well-posed since a continuous dependence estimate 
can be proved. Moreover, in the case when the two graphs exhibit the same growth,
it is shown
that the solution enjoys more regularity and the boundary condition holds almost 
everywhere. An error estimate can also be shown, for a suitable order of the diffusion parameter. 

\vspace{2mm}

\vspace{2mm}
\noindent \textbf{AMS (MOS) subject clas\-si\-fi\-ca\-tion:} 
35K61, 35K25, 35D30, 35B20 74N20, 
80A22.
\end{abstract}

%%%%% Section 1. %%%%%
\section{Introduction}
\setcounter{equation}{0}
We consider a pure Cahn--Hilliard equation in the form
\begin{align} 
	\partial_t u -\Delta \mu = 0 
	\qquad&\mbox{in } Q:=\Omega\times(0,T),
	\label{CH1}
	\\
	\mu = -\Delta u + F'(u)-f 
	\qquad&\mbox{in } Q,
	\label{CH2}
\end{align}
where $T>0$ is some fixed time, $\Omega \subset \mathbb{R}^{d}$ 
($d=2$ or $3$) is a bounded smooth domain 
with smooth enough boundary $\Gamma$,
and the symbols $\partial_t$ and $\Delta$ denote the 
partial time-derivative and the Laplacian with respect to the space variables, 
respectively. 

On the boundary, we deal with a dynamic condition also
of Cahn--Hilliard type, depending on a positive
parameter $\delta$, in the form 
\begin{align} 
	\label{trace}
	u_\Gamma =u _{|_\Gamma }, \quad \mu _\Gamma =\mu _{|_\Gamma }
	\qquad&\mbox{on } \Sigma:=\Gamma\times(0,T),\\
	\partial_t  u_\Gamma +\partial_\nu \mu -\Delta _\Gamma \mu _\Gamma =0
	\qquad&\mbox{on } \Sigma,
	\label{CHB1}
	\\
	\mu _\Gamma =
	\partial _\nu u - \delta\Delta _\Gamma u_\Gamma +F'_\Gamma (u_\Gamma )-f_\Gamma 
	\qquad&\mbox{on } \Sigma.
	\label{CHB2}
\end{align} 
Here, the notation $v_{|_\Gamma}$ is employed for the trace of a function $v :\Omega \to \mathbb{R} $ on the boundary $\Gamma$; besides,  $\partial_\nu$ and $\Delta_\Gamma$ denote the 
the outward normal derivative and the {L}aplace--{B}eltrami operator on $\Gamma$. 
At the initial time $t=0$, we assume that
\begin{align} 
	u(0)=u_0
	\qquad 
	&\mbox{in } \Omega, 
	\label{IC1}
	\\
	u_\Gamma  (0)=u_{0\Gamma}
	\qquad&\mbox{on }\Gamma.
	\label{IC2}
\end{align} 
The variables $u, \mu: Q \to \mathbb{R}$ and the respective ones  $u_\Gamma, \mu_\Gamma : \Sigma \to \mathbb{R}$ on the boundary represent 
the phase parameter and the chemical potential. Moreover, $f : Q \to \mathbb{R}$ and 
$f_\Gamma : \Sigma\to \mathbb{R}$ stand for two known source terms and $u_0, u_{0\Gamma} $ 
are the given initial data, in the bulk and on the boundary. The nonlinearities $F'$ and 
$F_\Gamma'$ owe their presence in the equations as derivatives of the double-well potentials $F$ and $F_\Gamma$.
 
This paper is dedicated to investigate the asymptotic behaviour of the system 
\eqref{CH1}--\eqref{IC2} as $\delta$ tends to $0$. The limit condition which is obtained 
on the boundary is rather interesting since, as we will explain later, it consists of a
forward-backward dynamic boundary condition.

From now, let us spend some words on the {C}ahn--{H}illiard system and
recall that it is a phenomenological model describing the spinodal decomposition in the framework of partial differential equations.  It originates from the work of J.~W.~Cahn~\cite{Ca61} and his collaboration with J.~E.~Hilliard~\cite{CH58}.
A number of research contributions in recent times has been devoted 
to Cahn--Hilliard and viscous Cahn--Hilliard~\cite{No88, NP} systems. 
An impressive amount of related references can be found in the literature, in particular 
we may refer to the review paper\cite{Mi17} and references therein. The coupling of  
{C}ahn--{H}illiard and other systems with dynamic boundary conditions turns out to be a 
research theme that has been developed quite intensively in the last twenties. If one considers a boundary dynamics of heat equation type, the well-posedness issue has been treated in \cite{RZ03} and a study of the convergence to equilibrium is shown in \cite{WZ04}. Since then, the {C}ahn--{H}illiard system with nonlinear equations as dynamic boundary condition (including the {A}llen--{C}ahn equation), was addressed from different viewpoints and studied in several papers: 
among other contributions we quote\takeshi{\cite{CF15, CF20, CFW, CGS14, CGS15, CGS17bis, CGS18, CGS18ter, FW21, FYW17, Gal07, GMS11, GMS09, LW19, OFFY20, Sca19}}. The \pier{articles} \cite{CF20bis, CGNS17, CGS17} refer instead to similar approaches but for different equations in the domain. Let us also mention the papers \cite{CGS15, CGS18bis, CS19} devoted to the analysis of optimal control problems for {some C}ahn--{H}illiard systems coupling equation and dynamic boundary condition. 
For completeness, let us also mention that vanishing diffusion studies 
on Cahn\takeshi{--}Hilliard and Allen\takeshi{--}Cahn equations have been pursued also 
in the case of stochastic forcing, for which we refer to \cite{Sca20}
and \cite{OS19}, respectively.

Therefore, the problem (\ref{CH1})--(\ref{IC2}) yields the 
{C}ahn--{H}illiard system in the bulk and on the boundary, where the 
variables on the boundary are the traces of the respective ones and equations
dynamic boundary conditions have the same structure as Cahn--Hillird systems. 
The related initial-boundary value problem has been investigated -- in the case $\delta >0$ -- in the paper~\cite{CF15} and also in~\cite{CGS18} when convection effects are taken into account. Both analyses allow the presence of singular and non-smooth potentials 
for $F$ and $F_\Gamma$. In fact, typical examples for these potentials
are the so-called classical regular potential, the logarithmic potential,
and the double obstacle potential, which are defined by
\begin{align*}
  & {F}_{\rm reg}(r) := \frac 14 \, (r^2-1)^2 \,,
  \quad r \in \mathbb{R}\, , 
  \\[2mm]
  &
  F_{\rm log}(r)
  := \left\{\begin{array}{ll}
    (1+r)\ln (1+r)+(1-r)\ln (1-r) - c_1 r^2\,,
    & \quad r \in (-1,1)
    \\[1mm]
    2\,{\ln}(2)-c_1\,,
    & \quad r\in\{-1,1\}
    \\[1mm]
    +\infty\,,
    & \quad r\not\in [-1,1]
  \end{array}\right. ,
  \\[2mm]
  & {F}_{\rm obs}(r) 
  := \left\{\begin{array}{ll}
    c_2 (1-r^2)\,,
    & \quad r \in [-1,1]
    \\[1mm]
    +\infty\,,
    & \quad r\not\in [-1,1]  
    \end{array}\right. ,  
\end{align*}
where 
$c_1>1$ and $c_2>0$ are constants, with the role of rendering 
${F}_{\rm log}$ and ${F}_{\rm obs}$ nonconvex.
Here, as in \cite{CF15, CGS18} we split the nonlinear contributions $F'$ in \eqref{CH2} and $F_\Gamma'$ in \eqref{CHB2} into two parts, i.e., we let $F'=\beta+\pi$ and $F_\Gamma'=\beta_\Gamma+\pi_\Gamma$,
where $\beta$, $\beta_\Gamma$ are the monotone parts, i.e. the derivatives, or in general the subdifferentials, of the convex parts of $F$ and $F_\Gamma$, and 
$\pi$, $\pi_\Gamma$ stand for the (smooth) anti-monotone parts.
In particular, for the classical regular potential 
$F_{\rm reg}'=\beta_{\rm reg}+\pi_{\rm reg}$
is exactly the derivative of $F_{\rm reg}$, that is
\begin{equation*}
	F'_{\rm reg}(r)=r^3-r, \quad \hbox{with} \quad \beta_{\rm reg}(r):=r^3, \quad \pi_{\rm reg}(r):=-r,
\end{equation*}
while for the non-smooth double obstacle potential $F_{\rm obs}$ we have that
$\beta_{\rm obs}$ is the subdifferential of the indicator function of $[-1,1]$, so to have 
\begin{equation*}
	F'_{\rm obs}(r)=\partial I_{[-1,1]}(r) - 2 c_2 r, \quad \beta_{\rm obs}(r):=\partial I_{[-1,1]}(r), \quad \pi_{\rm obs}(r):=-2 c_2r. 
\end{equation*}
As a general rule, we employ subdifferentials for $\beta$, $\beta_\Gamma$,
which reduce to the derivatives whenever these exist. 
Please note that the subdifferentials 
may also be multivalued graphs, as it happens for $\partial I_{[-1,1]} (r) $ when $r=-1$ or $r=1$. 
Thus, we generally interpret equation (\ref{CH2}) as
\begin{equation} 
	\mu = - \Delta u +\xi + \pi(u) -f, \quad 
	\xi \in \beta(u)
	\quad \text{in }Q
		\label{pi2}
\end{equation}
and the boundary condition (\ref{CHB2}) as
\begin{align} 
	&\mu_\Gamma = \partial_{\boldsymbol{\nu}} u  - \delta \Delta_\Gamma u_\Gamma 
	+\xi_\Gamma + \pi_\Gamma (u_\Gamma) -f_\Gamma,
	\quad  \xi_\Gamma \in \beta_\Gamma(u_\Gamma)
	\quad \text{on }\Sigma.
	\label{pib2}
\end{align}
Of course, different potentials can be considered for $F$ and $F_\Gamma$, leading
in particular to different graphs $\beta$ and $\beta_\Gamma$. 
About possible relations 
between $\beta$ and $\beta_\Gamma$, following a rather usual approach
(cf., e.g., \cite{CC13, CF15, CF20bis, CFW, CGNS17, CGS14, CGS17, CGS18, LW19, Sca19}) 
we assume in general 
that  $\beta_\Gamma$ dominates $\beta$ in the sense of assumption {\bf A2}  
in Section~2. For better regularity results (cf. the later 
Theorems~\ref{thm:conv2} and~\ref{thm:conv3}) we let $\beta$ and $\beta_\Gamma$
have the same growth (cf\takeshi{.} assumption~\eqref{ip_extra}). 
Thus, in our framework it is always 
possible to choose similar or even equal graphs $\beta$ and~$\beta_\Gamma$.

This paper is dedicated to the asymptotic analysis as the surface diffusion term on the dynamic boundary condition~(\ref{pib2}) tends to $0$. 
By the asymptotic limit as $\delta \searrow 0$, one aims to obtain at the limit 
the solution of the problem without surface diffusion, i.e., with \eqref{pib2} 
possibly replaced by 
\begin{align} 
	&\mu_\Gamma = \partial_{\boldsymbol{\nu}} u
	+\xi_\Gamma + \pi_\Gamma (u_\Gamma)
	-f_\Gamma,
	\quad  \xi_\Gamma \in \beta_\Gamma(u_\Gamma)
	\quad \text{on }\Sigma.
	\label{pib3}
\end{align}
It turns out that this program is doable, as shown by our analysis, by accepting that  
the solution of the limiting problem looses some regularity, due to the absence 
of the diffusive term in \eqref{pib3}. Indeed, in general the terms 
$\partial_{\boldsymbol{\nu}} u$ and $\xi_\Gamma$ in \eqref{pib2} are not 
functions but elements of a dual space, and the inclusion 
$ \xi_\Gamma \in \beta_\Gamma(u_\Gamma)$ has to be suitably 
reinterpreted in the sense of inclusion for a subdifferential 
operator acting from a space on the boundary to its dual space. 
However, the solution of the limiting problem turns out to be uniquely determined
at least for what concerns the components $(u,u_\Gamma) $. Moreover,  
in the case where the graphs $\beta$ and $\beta_\Gamma$
exhibit the same growth, we demonstrate that the boundary condition 
\eqref{pib3} holds almost everywhere on $\Sigma$. 
Besides that, in such a case we are even able to prove an error estimate 
of order $\delta^{1/2}$ between the solution of the problem 
with surface diffusion in \eqref{pib2} and that of the limiting problem with \eqref{pib3}.   

Two special references related to our investigation are the recent papers~\cite{Sca19} 
and \cite{CF15}, which deal with Cahn--Hilliard systems 
in the bulk with dynamic boundary condition of Allen--Cahn type, and extensions of them 
in \cite{Sca19} as well. In fact, our asymptotic results can be compared with the ones contained in 
these papers, where \cite{CF15} also examines the case of the graphs $\beta$ and $\beta_\Gamma$
having the same growth. However, no error estimate is discussed in \cite{CF15, Sca19} 
(as instead we do here).

The limiting equation \eqref{pib3} that 
we obtain on the boundary, when coupled to \eqref{CHB1},
yields a forward-backward type equation since the surface diffusion operator present in \eqref{CHB1}
eventually applies to $\mu_\gamma $ in \eqref{pib3}, and $\mu_\Gamma$ is given here in terms of a non-monotone 
function (i.e., $\beta_\Gamma+\pi_\Gamma$) of the phase variable $u_\Gamma$. About forward-backward 
equations and possible regularizations of them we can quote\takeshi{\cite{BCT17, BCST18, BCST20, CS16, Tomas, TST14}} and referenced therein. 

The main novelty of this paper is that we can give rigorous sense 
to a forward-backward dynamic on the boundary in terms of well-posedness
of the whole system. Indeed,
we underline that, unless special cases, 
the evolution problems for a single forward-backward equation are ill-posed.
Here, nonetheless, we show that the
coupling of the badly-behaving boundary condition
with the Cahn\takeshi{--}Hilliard equation in the interior
is somehow strong enough to ensure solvability of 
the boundary forward-backward dynamics.
We also point out that the Cahn\takeshi{--}Hilliard equation itself 
may be actually seen as an elliptic space-regularisation of a forward-backward 
equation by means of the local diffusion operator $-\Delta$.
Consequently, the choice of the limit forward-backward 
boundary condition is extremely natural, 
and corresponds to the intuitive degenerate limit 
of the Cahn\takeshi{--}Hilliard equation 
with no diffusion regularization on the boundary.

The present paper is structured as follows.
In Section~\ref{sec:main}, after setting up the notation and the basic tools for a precise interpretation
of the problem, we present the main theorems. First, we recall the well-posedness result for the case $\delta>0$; then, we state the convergence-existence result as $\delta $ goes to, and becomes 
$0$ at the limit, including the continuous dependence with respect to data and 
the uniqueness of the solution for the limit problem. There are two more statements, dedicated to 
an improvement of the convergence-existence theorem and to the error estimate in the case 
when the graphs $\beta$ and $\beta_\Gamma$ show the same growth. Section~\ref{sec:proof} is devoted to the proofs, in this order: we start with proving the uniform estimates for all $\delta \in (0,1)$, hence passing to the limit as $\delta \searrow 0$; then, we deal with the continuous dependence estimate, we examine the refined convergence and show the error estimate of order $\delta^{1/2}$. There is also a Section~\ref{appendix} with auxiliary results for equivalence of norms and approximation of initial data. 

%%%%% Section 2. %%%%%
\section{Setting and main results}
\setcounter{equation}{0}
\label{sec:main}

In this section, we rigorously introduce the 
variational setting and the main assumptions of the work,
and we state our main results.

Throughout the paper, $\Omega\subset\mathbb{R}^d$ ($d=2,3$)
is a smooth bounded domain with smooth boundary $\Gamma$,
and $T>0$ is a fixed finite final time.
We use the classical notations
\[
  Q_t:=\Omega\times(0,t), \quad \Sigma_t:=\Gamma\times(0,t)
  \quad\text{for } t\in[0,T], \qquad
  Q:=Q_T,\quad\Sigma:=\Sigma_T.
\]
We are interested in the asymptotic behaviour as $\delta\searrow0$ of the
following initial-boundary value problem:
\begin{align} 
	\partial_t u -\Delta \mu = 0 
	\qquad&\mbox{in } Q, 
	\label{eq1}
	\\
	\mu \in -\Delta u + \beta(u) + \pi(u) -f 
	\qquad&\mbox{in } Q,
	\label{eq2}\\
	u_\Gamma =u _{|_\Gamma }, \quad \mu _\Gamma =\mu _{|_\Gamma }
	\qquad&\mbox{on } \Sigma,\\
	\partial_t  u_\Gamma +\partial_\nu \mu -\Delta _\Gamma \mu _\Gamma =0
	\qquad&\mbox{on } \Sigma,
	\label{eq1_bound}
	\\
	\mu _\Gamma \in
	\partial _\nu u - \delta\Delta _\Gamma u_\Gamma +
	\beta_\Gamma (u_\Gamma ) + \pi_\Gamma(u_\Gamma) -f_\Gamma 
	\qquad&\mbox{on } \Sigma,
	\label{eq2_bound}\\
	u(0)=u_0
	\qquad 
	&\mbox{in } \Omega, 
	\label{eq1_init}
	\\
	u_\Gamma  (0)=u_{0\Gamma}
	\qquad&\mbox{on }\Gamma.
	\label{eq2_init}
\end{align} 

The following assumptions on the data are in order throughout the work.
\begin{description}
 \item[A1] $\widehat\beta,\widehat\beta _{\Gamma}:\erre\to[0,+\infty]$ are
 proper, convex, and lower semicontinous functions on $\erre$ satisfying the condition
 $\widehat\beta(0)=\widehat\beta_\Gamma(0)=0$. This implies that their subdifferentials
 \[
 \beta:=\partial\widehat\beta,\qquad
 \beta_\Gamma:=\partial\widehat\beta_\Gamma,
 \]
 are maximal monotone graphs in 
$\mathbb{R} \times \mathbb{R}$, 
with some effective domains $D(\beta )$ and $D(\beta _\Gamma)$, respectively,
and that $0 \in \beta (0)\cap\beta_\Gamma(0)$. 
 \item[A2] $D(\beta_\Gamma)\subseteq D(\beta)$ and there exists a constant $M>0$ such that
\beq\label{dom_beta}
  |\beta^\circ(r)|\leq M\left(1+|\beta_\Gamma^\circ(r)|\right)
  \quad\forall\,r\in D(\beta_\Gamma),
\eeq
where $\beta^\circ$ and $\beta_\Gamma^\circ$ denote the 
minimal sections of the graphs $\beta$ and $\beta_\Gamma$, respectively.
 \item[A3] $\pi $, $\pi _{\Gamma }: \mathbb{R} \to \mathbb{R}$ are {L}ipschitz continuous
with {L}ipschitz-constants $L$ and $L_{\Gamma}$, respectively, and we set 
\[
  \widehat\pi, \widehat\pi_\Gamma:\erre\to\erre, \qquad
  \widehat \pi(r):=\int_0^r\pi(s)\,ds, \quad
  \widehat \pi_\Gamma(r):=\int_0^r\pi_\Gamma(s)\,ds, \quad r\in\erre.
\]
We define for convenience of notation $\b\pi:=(\pi,\pi_\Gamma):\erre^2\to\erre^2$.
\end{description}

%%%%% Section 2.1. %%%%%
\subsection{Variational setting}
We describe here the variational setting that we consider and the 
concept of weak solution for the problem \eqref{eq1}--\eqref{eq2_init}.

We define the functional spaces
\begin{align*}
  &H:=L^2(\Omega), \quad V:=H^1(\Omega), \quad W:=H^2(\Omega),\\
  &H_\Gamma:= L^2(\Gamma), \quad
  Z_\Gamma:=H^{1/2}(\Gamma), \quad
  V_\Gamma:=H^1(\Gamma), \quad
  W_\Gamma:=H^2(\Gamma),
\end{align*}
endowed with their natural norms $\norm{\cdot}_{H}$,
$\norm{\cdot}_{V}$, $\norm{\cdot}_{W}$, 
$\norm{\cdot}_{H_\Gamma}$, $\norm{\cdot}_{Z_\Gamma}$,
$\norm{\cdot}_{V_\Gamma}$, $\norm{\cdot}_{W_\Gamma}$, and their 
scalar products $(\cdot,\cdot)_{H}$, $(\cdot,\cdot)_{V}$,
$(\cdot,\cdot)_{W}$,
 $(\cdot,\cdot)_{H_\Gamma}$
$(\cdot,\cdot)_{Z_\Gamma}$, $(\cdot,\cdot)_{V_\Gamma}$,
$(\cdot,\cdot)_{W_\Gamma}$.
We will denote by $z_{|\Gamma}$ the trace of 
the generic element $z\in V$.
Moreover, we set 
\begin{align*}
\bH&:=H\times H_\Gamma,\\
\bZ&:=\left\{ (z,z_\Gamma ) \in V \times Z_\Gamma \ : \  
z_\Gamma =z_{|_\Gamma} \text{ a.e.~on }\Gamma\right\},\\
\bV&:=\left\{ (z,z_\Gamma ) \in V \times V_\Gamma \ : \  
z_\Gamma =z_{|_\Gamma} \text{ a.e.~on }\Gamma\right\},\\
\b W&:=\left\{ (z,z_\Gamma ) \in W \times W_\Gamma \ : \  
z_\Gamma =z_{|_\Gamma} \text{ a.e.~on }\Gamma\right\}.
\end{align*}
Let us make clear now once and for all that we will use the bold 
notation $\mbox{\boldmath $z$}=(z,z_\Gamma)$ for the generic element in $\bH$.
Note that if $\mbox{\boldmath $z$}\in\bH$, then $z_\Gamma$ is not necessarily 
the trace of $z$ on the boundary: this is true only if at least $\mbox{\boldmath $z$}\in\bZ$.
Clearly, $\bH$, $\bZ$, $\bV$, and $\b W$ are Hilbert spaces 
with respect to the scalar products 
\begin{alignat*}{2}
	(\mbox{\boldmath $w$},\mbox{\boldmath $z$}
	)_{\mbox{\scriptsize \boldmath $ H$}}
	&:=(u,z)_{H} + (u_\Gamma ,z_\Gamma )_{H_\Gamma },
	\qquad&&\mbox{\boldmath $w$},\mbox{\boldmath $z$}\in\bH,\\
	(\mbox{\boldmath $ u $},\mbox{\boldmath $ z $}
	)_{\mbox{\scriptsize \boldmath $ Z$}}
	&:=(u,z)_{V} + (u_\Gamma ,z_\Gamma )_{Z_\Gamma },
	\qquad&&\mbox{\boldmath $w$},\mbox{\boldmath $z$}\in\bZ,\\
	(\mbox{\boldmath $ u $},\mbox{\boldmath $ z $}
	)_{\mbox{\scriptsize \boldmath $ V$}}
	&:=(u,z)_{V} + (u_\Gamma ,z_\Gamma )_{V_\Gamma },
	\qquad&&\mbox{\boldmath $w$},\mbox{\boldmath $z$}\in\bV,\\
	(\mbox{\boldmath $ u $},\mbox{\boldmath $ z $}
	)_{\mbox{\scriptsize \boldmath $ W$}}
	&:=(u,z)_{W} + (u_\Gamma ,z_\Gamma )_{W_\Gamma },
	\qquad&&\mbox{\boldmath $w$},\mbox{\boldmath $z$}\in\b W,
\end{alignat*}
and the respective norms $\norm{\cdot}_{\bsH}$, 
$\norm{\cdot}_{\bsZ}$, $\norm{\cdot}_{\bsV}$, and $\norm{\cdot}_{\bsW}$.
The Hilbert space $\bH$ is identified to its dual through 
the Riesz isomorphism, so that we have the 
continuous and dense embeddings
\[
  \bW\embed \bV \embed \bZ \embed \bH \embed \bV^*,
\]
where the inclusions $\bW\embed \bV\embed \bH$, $\bZ\embed \bH$, and $\bH\embed \bV^*$
are also compact.

We introduce the generalized ``mean'' operator $m:\bV^*\to\erre$ as 
\[
  m(\b z):=\frac{1}{|\Omega|+|\Gamma|}\ip{\b z}{\b 1}_{\bsV^*,\bsV}=
  \frac{1}{|\Omega|+|\Gamma|}\takeshi{{}\bigl({}}\ip{z}{1}_{V^*,V} + \ip{z_\Gamma}{1}_{V_\Gamma^*,V_\Gamma}\takeshi{{}\bigr){}},
   \quad \b z\in \bV^*,
\]
and define the subspace of null-mean elements as
\[
  \bH_0:=\bH\cap\operatorname{ker}(m), \qquad
  \bV_0:=\bV\cap\bH_0, \qquad 
  \takeshi{\bZ_0:=\bZ\cap\bH_0,}
\]
endowed with the norms 
\begin{align*}
  \norm{\b z}_{\bsH_0}&:=\norm{\b z}_{\bsH}, \quad \b z \in \bH,\\
  \norm{\b z}_{\bsV_0}&:=\left(\norm{\nabla z}^2_H + 
  \norm{\nabla_\Gamma z_\Gamma}^2_{H_\Gamma}\right)^{1/2}, \quad \b z \in \bV_0\takeshi{.}
\end{align*}
Let us recall the following \takeshi{Poincar\'e-type} inequalities:
\begin{alignat}{2}
  \label{poin1}
  \exists\,C_p>0:\quad
  &\norm{\b z}_{\bsV}\leq C_p\norm{\b z}_{\bsV_0} \quad&&\forall\,\b z
  =(z,z_\Gamma)\in\bV_0,\\
  \label{poin2}
  \exists\,C_p>0:\quad
  &\norm{z}_{V}\leq C_p\norm{\nabla z}_{H} \quad&&\forall\,\b z=(z,z_\Gamma)\in\bZ_0.
\end{alignat}
For the proof of \eqref{poin1} the reader can refer to \cite[Lem.~A]{CF15},
while the proof of \eqref{poin2} is given in Lemma~\ref{lem:A1} in the Appendix.
These imply in particular 
that an equivalent norm in the space $\bV$ is given by
\beq
  \label{norm1}
  \b z\mapsto 
  \left(\norm{\b z -m(\b z)\b 1}^2_{\bsV_0} + |m(\b z)|^2\right)^{1/2}, \quad \b z\in\bV,
\eeq
while an equivalent norm in $V$ is given by 
\beq
  \label{norm2}
  z\mapsto \left(\norm{\nabla z}^2_{H} + |m(z, z_{|\Gamma})|^2\right)^{1/2}, \quad z\in V.
\eeq
Moreover, we define the linear operator 
\[
  \mathcal L:\bV\to\bV^*, \qquad
  \ip{\mathcal L \b v}{\b z}_{\bsV^*, \bsV}:=
  \int_\Omega\nabla v\cdot\nabla z + 
  \int_\Gamma\nabla_\Gamma v_\Gamma\cdot\nabla_\Gamma z_\Gamma, 
  \quad \b v, \b z\in \bV,
\]
and note that $\mathcal L\b 1=\takeshi{\b0}$ \takeshi{in $\bV^*$}: hence, since $\b V=\b V_0\oplus\text{span}\{\b 1\}$,
we have that the restriction of $\mathcal L$ to $\b V_0$
is injective and with range 
\[
  \b V_{0,*}:=\mathcal L(\b V_0)=\{\b z\in \b V^*:\; m(\b z)=0\}.
\]
Consequently, $\mathcal L:\b V_0 \to \b V_{0,*}$
is a linear isomorphism with well-defined inverse $\mathcal L^{-1}:\bV_{0,*}\to\bV_0$.
With this notation, we introduce the norm 
\[
\takeshi{
  \|\b z\|_*:=\takeshi{\bigl(} \|\mathcal L^{-1}(\b z- m(\b z))\|_{\bsV_0}^2 + |m(\b z)|^2 \takeshi{\bigr)^{1/2}},
  \qquad \b z\in \b V^*,
}
\]
which is equivalent to the usual dual norm on $\b V^*$ and satisfies
\beq\label{chain}
  \langle\partial_t \b z, \mathcal L^{-1}\b z\rangle_{\bsV^*, \bsV}
  =\frac{d}{dt}\frac12\|\b z\|_*^2
  \qquad\forall\,\b z \in H^1(0,T; \b V_{0,*}).
\eeq

%%%%% Section 2.1. %%%%%
\subsection{Concepts of solution}
Let us precise here the concepts of variational (weak) solution
for the system \eqref{eq1}--\eqref{eq2_init}, in the cases $\delta \in (0,1)$
and $\delta=0$,
respectively.

\begin{defin}[$\delta>0$]
  \label{sol:delta_pos}
  Let $\delta>0$, and 
  \[
  \b u_{0}^\delta\in \bV,\qquad \b f^\delta\in L^2(0,T; \bH).
  \]
  A weak solution to the problem \eqref{eq1}--\eqref{eq2_init} is a triplet
  $(\b u^\delta, \b\mu^\delta,\b\xi^\delta)$, with
  \begin{align*}
  &\b u^\delta \in H^1(0,T; \bV^*)\cap L^\infty(0,T; \bV)\cap L^2(0,T;\b W),\\
  &\b\mu^\delta \in L^2(0,T; \bV),\\
  &\b\xi^\delta\in L^2(0,T;\bH),
  \end{align*}
  such that $\b u^\delta(0)=\b u_0^\delta$,
  \begin{align}
  \label{var1_delta_pos}
  \langle\partial_t\b u^\delta,\b z\rangle_{\bsV^*,\bsV}
  +\int_\Omega\nabla\mu^\delta\cdot\nabla z
  +\int_\Gamma\nabla_\Gamma\mu_\Gamma^\delta\cdot\nabla_\Gamma z_\Gamma=0
  \qquad&\forall\,\b z\in\bV,\\
  \label{var2_delta_pos}
  (\b\mu^\delta,\b z)_{\bsH}=
  \int_\Omega\nabla u^\delta\cdot\nabla z
  +\delta\int_\Gamma\nabla_\Gamma u_\Gamma^\delta\cdot\nabla_\Gamma z_\Gamma
  +(\b\xi^\delta+\b\pi(\b u^\delta) - \b f^\delta, \b z)_{\bsH}
  \qquad&\forall\,\b z\in\bV,
  \end{align}
  almost everywhere in $(0,T)$, and
  \begin{align}
  \label{incl1_delta_pos}
  \xi^\delta\in\beta(u^\delta)\qquad&\text{a.e.~in } Q,\\
  \label{incl2_delta_pos}
  \xi_\Gamma^\delta\in\beta_\Gamma(u_\Gamma^\delta)\qquad&\text{a.e.~\takeshi{on} } \Sigma.
  \end{align}
\end{defin}

\begin{defin}[$\delta=0$]
  \label{sol:delta0}
  Let $\delta=0$, and
  \[
  \b u_{0}\in \bZ,\qquad \b f\in L^2(0,T; \bH).
  \]
  A weak solution to the problem \eqref{eq1}--\eqref{eq2_init} is a triplet
  $(\b u, \b\mu,\b\xi)$, with
  \begin{align*}
  &\b u \in H^1(0,T; \bV^*)\cap L^\infty(0,T; \bZ),\qquad \Delta u\in L^2(0,T; H)\\
  &\b\mu \in L^2(0,T; \bV),\\
  &\b\xi\in L^2(0,T;H\times Z_\Gamma^*),\\
  \end{align*}
  such that $\b u(0)=\b u_0$, 
  \begin{align}
  \label{var1_delta0}
  \langle\partial_t\b u,\b z\rangle_{\bsV^*,\bsV}
  +\int_\Omega\nabla\mu\cdot\nabla z
  +\int_\Gamma\nabla_\Gamma\mu_\Gamma\cdot\nabla_\Gamma z_\Gamma=0
  \qquad&\forall\,\b z\in\bV,\\
  \label{var2_delta0}
  (\b\mu,\b z)_{\bsH}=
  \int_\Omega\nabla u\cdot\nabla z
  +(\xi,z)_H+\langle\xi_\Gamma,z_\Gamma\rangle_{Z_\Gamma^*,Z_\Gamma}
  +(\b\pi(\b u) - \b f, \b z)_{\bsH}
  \qquad&\forall\,\b z\in\bZ,
  \end{align}
  almost everywhere in $(0,T)$, and 
  \begin{align}
  \label{incl1_delta0}
  &\xi\in\beta(u)\qquad\text{a.e.~in } Q,\\
  \label{incl2_delta0}
  &\int_\Sigma\widehat\beta_\Gamma(u_\Gamma)+
  \int_0^T
  \langle\xi_\Gamma,z_\Gamma-u_\Gamma\rangle_{Z_\Gamma^*,Z_\Gamma}\leq
  \int_\Sigma\widehat\beta_\Gamma(z_\Gamma) 
  \quad\forall\,z_\Gamma\in L^2(0,T; Z_\Gamma),
  \end{align}
  where the last integral is intended to be $+\infty$ whenever 
  $\widehat\beta_\Gamma(z_\Gamma)\notin L^1(\Sigma)$.
\end{defin}

\begin{remark}\rm
  Let us point out that the variational equalities \eqref{var1_delta_pos} and 
  \eqref{var1_delta0} can be formally obtained from the 
  equations \eqref{eq1} and \eqref{eq1_bound} multiplying by 
  the generic pair $(z,z_\Gamma)\in\b V$ and integrating by parts.
  As a matter of fact, 
  the equations \eqref{var1_delta_pos} and 
  \eqref{var1_delta0} actually 
  provide a representation of the time derivative $\partial_t\b u^\delta$
  and $\partial_t\b u$ as elements of the dual space $L^2(0,T; \b V^*)$.
\end{remark}

\begin{remark}\rm
\label{rem1}
We note that the variational equation \eqref{var2_delta_pos}
entails 
\begin{align}
  \label{pier1}
  \mu^\delta =  -\Delta u^\delta + \xi^\delta + \pi(u^\delta) - f^\delta
  \qquad&\text{a.e.~in } Q,\\
  \label{pier2}
  \mu_\Gamma^\delta = \partial_\nu u^\delta - \delta\Delta_\Gamma u_\Gamma^\delta
  +\xi^\delta_\Gamma + \pi_{\takeshi{\Gamma}}(u^\delta_\Gamma) - f^\delta_\Gamma
  \qquad&\text{a.e~on } \Sigma.
\end{align}
Indeed, \eqref{pier1} follows testing \eqref{var2_delta_pos} by
the generic pair $(z,0)$ with $z\in H^1_0(\Omega)$, integrating by parts, 
and using the regularity of $u^\delta$. Then, due to the regularity of 
$\partial_\nu u^\delta$ and
$u^\delta_\Gamma$, the boundary condition \eqref{pier2} can be easily 
derived from \eqref{var2_delta_pos} using \eqref{pier1}.
\end{remark}

\begin{remark}\rm
In the same spirit, one can argue on \eqref{var2_delta0}
to deduce 
\begin{align}
  \label{pier3}
  \mu =  -\Delta u + \xi + \pi(u) - f
  \qquad&\text{a.e.~in } Q,\\
  \label{pier4}
  \mu_\Gamma = \partial_\nu u 
  +\xi_\Gamma + \pi_{\takeshi{\Gamma}}(u_\Gamma) - f_\Gamma
  \qquad&\text{in } Z_\Gamma^*,\text{ a.e.~in } (0,T).
\end{align}
Indeed, \eqref{pier3} follows as above 
testing  \eqref{var2_delta0} by $(z,0)$ with $z\in H^1_0(\Omega)$,
using integration by parts,
and the regularity of $\Delta u$. As for \eqref{pier4}, 
since for almost every $t\in(0,T)$ it holds that $u(t)\in V$
and $\Delta u(t)\in H$, by \cite[Thm.~2.27, p.~1.64]{BG87}
we have that $\partial_\nu u(t)$ is well-defined in 
$Z_\Gamma^*\cong H^{-1/2}(\Gamma)$. Hence,
\eqref{pier4} can be deduced from \eqref{var2_delta0}
using \eqref{pier3}.
\end{remark}

\begin{remark}\rm
  \label{rem3}
  Let us comment on condition \eqref{incl2_delta0}.
  Whenever $\xi_\Gamma \in L^2(0,T; H_\Gamma)$, it turns out that
  \eqref{incl2_delta0} is actually equivalent to the classical inclusion
  \[
  \xi_\Gamma \in \beta_\Gamma(u_\Gamma) \qquad\text{a.e.~on } \Sigma,
  \]
  or equivalently
  \[
  \xi_\Gamma \in \partial I_\Sigma(u_\Gamma),
  \]
  where
  \[
  I_\Sigma:L^2(0,T; H_\Gamma)\to[0,+\infty],
  \qquad
  I_\Sigma(z_\Gamma):=
  \begin{cases}
  \int_\Sigma\widehat\beta_\Gamma(z_\Gamma) \quad&\text{if } 
  \widehat\beta_\Gamma(z_\Gamma) \in L^1(\Sigma),\\
  +\infty\quad&\text{otherwise}.
  \end{cases}
  \]  
  More generally, if we only have $\xi_\Gamma \in L^2(0,T; Z_\Gamma^*)$, then
  \eqref{incl2_delta0} means that 
  \[
  \xi_\Gamma \in \partial J_\Sigma(u_\Gamma),
  \]
  where
  \[
  J_\Sigma:L^2(0,T; Z_\Gamma)\to[0,+\infty],
  \qquad
  J_\Sigma(z_\Gamma):=
  \begin{cases}
  \int_\Sigma\widehat\beta_\Gamma(z_\Gamma) \quad&\text{if } 
  \widehat\beta_\Gamma(z_\Gamma) \in L^1(\Sigma),\\
  +\infty\quad&\text{otherwise}.
  \end{cases}
  \]
  Here, the main point is that,
  since we are identifying $H_\Gamma$ to its dual,
  the subdifferential $\partial I_\Sigma$
  is intended as a multivalued operator 
  \[
  \partial I_\Sigma:L^2(0,T; H_\Gamma) \to 2^{L^2(0,T; H_\Gamma)},
  \]
  while $\partial J_\Sigma$ is seen as an operator 
  \[
  \partial J_\Sigma:L^2(0,T; Z_\Gamma) \to 2^{L^2(0,T; Z^*_\Gamma)}.
  \]
  For further details we refer to \cite{Bar10, Bre73}.
\end{remark}

\subsection{Main results}
Let us recall the well-posedness result for the system \eqref{eq1}--\eqref{eq2_init}
when $\delta>0$ is fixed: the reader can refer to \cite[Thm.~2.1--2.2]{CF15}.
\begin{thm}
  \label{thm:WP}
  Assume {\bf A1--A3}, let $\delta>0$ be fixed, and suppose that
  \begin{align}
  \label{ip1_delta_pos}
   &\b u_{0}^\delta\in \bV, \qquad
   \widehat\beta(u_0^\delta)\in L^1(\Omega), \qquad
   \widehat\beta_\Gamma(u_{0\Gamma}^\delta)\in L^1(\Gamma), \qquad
   m(\b u_0^\delta) \in \operatorname{Int}D(\beta_\Gamma),\\
   \label{ip2_delta_pos}
   &\b f^\delta:=\b g^\delta + \b h^\delta, \qquad
   \b g^\delta\in W^{1,1}(0,T; \bH), \qquad
   \b h^\delta\in L^2(0,T; \bV).
  \end{align}
  Then, there exists a weak solution $(\b u^\delta, \b\mu^\delta, \b\xi^\delta)$
  of the system \eqref{eq1}--\eqref{eq2_init}, in the sense of Definition~\ref{sol:delta_pos}.
  Moreover, there exists a constant $K_\delta>0$ such that, 
  for any data $\{(\b u^\delta_{0,i}, \b f^\delta_i)\}_{i=1,2}$ satisfying 
  \eqref{ip1_delta_pos}--\eqref{ip2_delta_pos} and 
  $m(\b u^\delta_{0,1})=m(\b u^\delta_{0,1})$, any respective weak solutions 
  $\{(\b u^\delta_i, \b\mu^\delta_i, \b\xi^\delta_i)\}_{i=1,2}$ satisfy
  \[
  \|\b u^\delta_1-\b u^\delta_2\|_{L^\infty(0,T;\bsV^*)\cap L^2(0,T; \bsV)}
  \leq K_\delta\left(\|\b u_{0,1}^\delta - \b u_{0,2}^\delta\|_{\bsV^*} + 
  \|\b f^\delta_1-\b f^\delta_2\|_{L^2(0,T; \bsV^*)}\right).
  \]
  In particular, the solution components $\b u^\delta$ and $\b\mu^\delta-\b\xi^\delta$
  are unique. If also $\beta$ is single-valued then the whole
  triplet $(\b u^\delta, \b\mu^\delta, \b\xi^\delta)$ is unique as well.
\end{thm}

We are now ready to state our main results.
\begin{thm}
  \label{thm:conv}
  Assume {\bf A1--A3} and let 
  \begin{align}
  \label{ip1_delta0}
    &\b u_0 \in\bZ, \qquad
    \widehat\beta(u_0)\in L^1(\Omega), \qquad
   \widehat\beta_\Gamma(u_{0\Gamma})\in L^1(\Gamma), \qquad
   m(\b u_0) \in \operatorname{Int}D(\beta_\Gamma),\\
   \label{ip2_delta0}
   &\b f :=\b g + \b h, \qquad
   \b g\in W^{1,1}(0,T; \bH), \qquad
   \b h\in L^2(0,T; \bV).
  \end{align}
  Consider 
  a family of data $\{(\b u_{0}^\delta, \b f^\delta)\}_{\delta\in(0,1)}$ which 
  satisfy assumptions~\eqref{ip1_delta_pos}--\eqref{ip2_delta_pos}
  and denote by $\{(\b u^\delta, \b \mu^\delta, \b\xi^\delta)\}_{\delta\in(0,1)}$ 
  the 
  respective weak solutions of the system \eqref{eq1}--\eqref{eq2_init}
  given by Theorem~\ref{sol:delta_pos}. Suppose also that 
  there exists a constant $M_0>0$ such that 
  \begin{align}
  \label{ip1}
  \delta\|\nabla_\Gamma u_{0\Gamma}^\delta\|_{H_\Gamma}^2
  +\big\|\widehat\beta(u_0^\delta)\big\|_{L^1(\Omega)}
   +\big\|\widehat\beta_\Gamma(u_{0\Gamma}^\delta)\big\|_{L^1(\Gamma)}
  \leq M_0
  \qquad&\forall\,\delta\in(0,1),\\
  \label{ip2}
  \|\b g^\delta\|_{W^{1,1}(0,T; \bsH)} +
  \|\b h^\delta\|_{L^2(0,T; \bsV)}
  \leq M_0
  \qquad&\forall\,\delta\in(0,1),
  \end{align}
  and that, as $\delta\to0$,
  \beq\label{ip3}
  \b u_0^\delta \wto \b u_0 \quad\text{in } \bZ, \qquad
  \b f^\delta \wto \b f \quad\text{in } L^2(0,T;\bH).
  \eeq
  Then, there exists a weak solution 
  $(\b u, \b\mu, \b\xi)$
  of the system \eqref{eq1}--\eqref{eq2_init} with $\delta=0$
  in the sense of Definition~\ref{sol:delta0}, such that,
  as $\delta\to0$, 
  \begin{align}
  \label{c1}
  \b u^\delta \to \b u \quad&\text{in } C^0([0,T]; \bH),\\
  \b u^\delta \wstarto \b u \quad&\text{in } H^1(0,T; \bV^*)\cap L^\infty(0,T;\bZ),\\
  \Delta u^\delta \wto \Delta u \quad&\text{in } L^2(0,T; H),\\
  \b \mu^\delta \wto \b \mu \quad&\text{in }  L^2(0,T;\bV),\\
  \b \xi^\delta \wto \b \xi \quad&\text{in } L^2(0,T; H\times V_\Gamma^*),\\
  -\delta\Delta_\Gamma u_\Gamma^\delta + \xi_\Gamma^\delta \wto\xi_\Gamma
  \quad&\text{in } L^2(0,T; Z_\Gamma^*),\\
  \delta \b u^\delta \to \takeshi{\b0} \quad&\text{in } L^\infty(0,T; \bV).
  \label{c7}
  \end{align}
  Moreover, there exists a constant $K>0$ such that, 
  for any data $\{(\b u_{0,i}, \b f_i)\}_{i=1,2}$ satisfying 
  \eqref{ip1_delta0}--\eqref{ip2_delta0} and 
  $m(\b u_{0,1})=m(\b u_{0,1})$, any respective weak solutions 
  $\{(\b u_i, \b\mu_i, \b\xi_i)\}_{i=1,2}$ 
  of \eqref{eq1}--\eqref{eq2_init} with $\delta=0$
  in the sense of Definition~\ref{sol:delta0}
  satisfy
  \[
  \|\b u_1-\b u_2\|_{L^\infty(0,T;\bsV^*)\cap L^2(0,T; \bsZ)}
  \leq K\left(\|\b u_{0,1} - \b u_{0,2}\|_{\bsV^*} + 
  \|\b f_1-\b f_2\|_{L^2(0,T; \bsH)}\right).
  \]
  In particular, the solution components $\b u$ and $\b\mu-\b\xi$
  are unique. If also $\beta$ is single-valued then the whole
  triplet $(\b u, \b\mu, \b\xi)$ is unique as well.
\end{thm}

\begin{remark}\rm
The existence of an approximating sequence $\{\b u_0^\delta\}_{\delta\in(0,1)}$
satisfying \eqref{ip1} is discussed in the Appendix under the additional  
assumption \eqref{ip_extra} specified below. In the general case, 
we point out that the most natural choice for $\{\b u_0^\delta\}_{\delta\in(0,1)}$
is given by the constant sequence $\b u_0$ in the case 
$\b u_0\in \b V$. Similarly, a typical choice for 
$\{\b f^\delta\}_{\delta\in(0,1)}$ is the constant one $\b f$.
\end{remark}

\begin{thm}
  \label{thm:conv2}
  In the setting of Theorem~\ref{thm:conv}, suppose also that 
  \begin{align}\label{ip_extra}
  &D(\beta)=D(\beta_\Gamma), \quad \hbox{there exists a 
  constant $M\geq1$ such that} 
  \nonumber\\
  &\quad{}\frac1M|\beta_\Gamma^\circ(r)| - M\leq |\beta^\circ(r)| \leq 
  M (|\beta_\Gamma^\circ(r)|+1)
  \quad\forall\,r\in D(\beta).
  \end{align}
  Then, the limiting triplet $(\b  u, \b\mu, \b\xi)$ 
  obtained in Theorem~\ref{thm:conv}
  also satisfies
  \begin{align*}
  &u\in L^2(0,T; H^{3/2}(\Omega)), \qquad u_\Gamma\in L^2(0,T; V_\Gamma),
  \qquad\partial_\nu u \in L^2(0,T; H_\Gamma),\\
  &\xi_\Gamma\in L^2(0,T; H_\Gamma), \qquad
  \xi_\Gamma\in\beta_\Gamma(u_\Gamma) \quad\text{a.e.~in } \Sigma,
  \end{align*}
  and, in addition to \eqref{c1}--\eqref{c7}, the following convergences hold:
  \begin{align*}
  \b\xi^\delta \wto \b\xi \quad&\text{in } L^2(0,T; \bH),\\
  \delta u_\Gamma^\delta \wto 0
   \quad&\text{in } L^2(0,T; H^{3/2}(\Gamma)),\\
  \partial_\nu u^\delta - \delta \Delta_\Gamma u_\Gamma^\delta \wto \partial_\nu u
   \quad&\text{in } L^2(0,T; H_\Gamma),
  \end{align*}
  In particular, \eqref{pier4} entails
\begin{align}
  \label{pier5}
  \mu_\Gamma = \partial_\nu u 
  +\xi_\Gamma + \pi_{{\takeshi \Gamma}}(u_\Gamma) - f_\Gamma
  \qquad&\text{a.e.~on } \Sigma.
\end{align}
\end{thm}

\begin{remark}\rm
  We note that the additional regularity in Theorem~\ref{thm:conv2}
  in particular implies that $\partial_\nu u, \xi_\Gamma$ are
  well-defined in $H_\Gamma$, almost  everywhere in $(0,T)$. 
  Consequently, equation \eqref{pier4} holds not only in $Z_\Gamma^*$,
  but also in $H_\Gamma$, almost everywhere in $(0,T)$, and this
  directly implies the validity of \eqref{pier5}.
\end{remark}

\begin{thm}
  \label{thm:conv3}
  In the setting of Theorem~\ref{thm:conv2},
  assume that $m(\b u_0^\delta)=m(\b u_0)$
  for all $\delta\in(0,1)$. Then,
  there exists a constant $C>0$, independent of $\delta$, such that 
  \[
  \|\b u^\delta - \b u\|_{L^\infty(0,T;\bsV^*)\cap L^2(0,T; \bsZ)}
  \leq C\left(\delta^{1/2} +
  \|\b u_0^\delta - \b u_0\|_{\bsV^*} + \|\b f^\delta - \b f\|_{L^2(0,T; \bsH)}\right)
  \]
  for every $\delta \in (0,1)$ and, as $\delta\searrow0$,
  \[
  \b u^\delta \wto \b u \qquad\text{in } L^2(0,T; \b V).
  \]
  In particular, if
  \[
  \|\b u_0^\delta - \b u_0\|_{\bsV^*} + \|\b f^\delta - \b f\|_{L^2(0,T; \bsH)}
  =O(\delta^{1/2})
  \qquad\text{as } \delta\searrow0,
  \]
  then
   \[
  \|\b u^\delta - \b u\|_{L^\infty(0,T;\bsV^*)\cap L^2(0,T; \bsZ)} = O(\delta^{1/2})
  \qquad\text{as } \delta\searrow 0.
  \]
\end{thm}

%%%%%%%%%%%%%

\section{Proofs}
\setcounter{equation}{0}
\label{sec:proof}
This section is devoted to
proving Theorems~\ref{thm:conv}, \ref{thm:conv2}, and \ref{thm:conv3}.

\subsection{Uniform estimates}
\noindent{\bf First estimate.} Testing equation
\eqref{var1_delta_pos} by \takeshi{$\b1/(|\Omega|+|\Gamma|)$}
and integrating in time we get 
\beq
  \label{mass}
  m(\b u^\delta(t)) = m(\b u_0^\delta) \qquad\forall\,t\in[0,T].
\eeq
By assumption \eqref{ip3}
on the initial data \takeshi{$\{ \b u^\delta_0 \}_{\delta\in (0,1)}$}, 
it holds that $m(\b u^\delta_0)\to m({\takeshi {\b u}_0})$ as $\delta\searrow0$.
Hence, from \eqref{ip1_delta_pos}, \eqref{ip1_delta0}, and
assumption {\bf A2} it follows that
\beq
  \label{interval}
  \exists\,[a,b]\subset\operatorname{Int}D(\beta_\Gamma)
  \subseteq\operatorname{Int}D(\beta):\quad
  m(\b u_0^\delta)\in[a,b] \quad\forall\,\delta \in (0,1).
\eeq
We deduce that
there exists a constant $C>0$, independent of $\delta$, such that 
\beq
  \label{est0}
  \|m(\b u^\delta)\|_{L^\infty(0,T)}\leq C.
\eeq
In the same spirit, equation \eqref{var1_delta_pos} 
directly implies by comparison \takeshi{and the \pier{Schwarz} inequality} that 
\beq
  \label{est00}
  \|\partial_t\b u^\delta(t)\|_{\bsV^*}\leq
  \|\nabla\mu^\delta(t)\|_H + \|\nabla_\Gamma\mu_\Gamma(t)\|_{H_\Gamma}
  \qquad\text{for a.e.~$t\in(0,T)$}.
\eeq

\pier{
\noindent{\bf Second estimate.} 
We first note that by \eqref{mass} we have that 
$\b u^\delta - m(\b u_0^\delta)\b 1\in \b V_0$ in $[0,T]$. Hence, 
we can test equation \eqref{var1_delta_pos} by 
$\mathcal L^{-1}(\b u^\delta - m(\b u_0^\delta)\b 1)$, 
equation \eqref{var2_delta_pos} by $-(\b u^\delta - m(\b u_0^\delta)\b 1)$, 
and sum. By doing this, we note that 
there is a cancellation of two terms since for 
$\b z=\mathcal L^{-1}(\b u^\delta - m(\b u_0^\delta)\b 1)$ we have
\begin{align*}
  &\int_\Omega\nabla\mu^\delta\cdot\nabla z
  +\int_\Gamma\nabla_\Gamma\mu_\Gamma^\delta\cdot 
  \nabla_\Gamma z_\Gamma \\
  &=
  \int_\Omega\nabla(\mu^\delta-m(\b\mu^\delta))\cdot\nabla z
  +\int_\Gamma\nabla_\Gamma(\mu_\Gamma^\delta-m(\b\mu^\delta))\cdot
  \nabla_\Gamma z_\Gamma \\
  &  =\langle\mathcal L(\b\mu^\delta-m(\b\mu^\delta)\b 1),\b z\rangle_{\bsV^*, \bsV}
  =(\b\mu^\delta-m(\b\mu^\delta)\b 1, \b u^\delta-m(\b u^\delta)\b 1)_{\bsH}\\
  &=(\b\mu^\delta, \b u^\delta-m(\b u^\delta)\b 1)_{\bsH}.
\end{align*}
Hence, we obtain
\begin{align*}
  &\takeshi{ \bigl \langle} \partial_t \b u^\delta, \mathcal L^{-1}(\b u^\delta 
  - m(\b u_0^\delta)\b 1) \takeshi{\bigr \rangle}_{\bsV^*, \bsV}
  +\int_\Omega|\nabla u^\delta|^2
  +\delta\int_\Gamma|\nabla_\Gamma u_\Gamma^\delta|^2\\
  \nonumber
  &\qquad+\int_\Omega\xi^\delta(u^\delta - m(\b u_0^\delta))
  +\int_\Omega\xi_\Gamma^\delta(u_\Gamma^\delta - m(\b u_0^\delta))
  \\
  &=(\takeshi{\b f^\delta -\b \pi (\b u^\delta)},\b u^\delta - m(\b u_0^\delta)\b 1)_{\bsH}.
\end{align*}
Thanks to the remark \eqref{interval}
we can use the inequalities devised by Miranville and Zelik \cite{MZ04}
(a proof can be checked also in \cite[\S~5]{GMS09}) to infer that
there is $C_{MZ}>0$ such that
\begin{align*}
  \int_\Omega\xi^\delta(u^\delta - m(\b u_0^\delta))
  +\int_{\takeshi{\Gamma}}\xi_\Gamma^\delta(u_\Gamma^\delta - m(\b u_0^\delta))
  \geq 
  C_{MZ}\left(\|\xi^\delta\|_{L^1(\Omega)} +
   \|\xi_\Gamma^\delta\|_{L^1(\Gamma)}\right) - C
\end{align*}
almost everywhere in $(0,T)$,
from which we obtain 
\begin{align}
  \nonumber
  &\bigl \langle \partial_t \b u^\delta, \mathcal L^{-1}(\b u^\delta 
  - m(\b u_0^\delta)\b 1) \takeshi{\bigr \rangle}_{\bsV^*, \bsV}
  +\int_\Omega|\nabla u^\delta|^2
  +\delta\int_\Gamma|\nabla_\Gamma u_\Gamma^\delta|^2\\
  \nonumber
  &\qquad+C_{MZ}\left(\|\xi^\delta\|_{L^1(\Omega)} +
   \|\xi_\Gamma^\delta\|_{L^1(\Gamma)}\right)
  \\
  &\leq C + (\b f^\delta -\b \pi (\b u^\delta),\b u^\delta - m(\b u_0^\delta)\b 1)_{\bsH}.
  \label{est_aux}
\end{align}
We now integrate \eqref{est_aux} in time
using the chain rule \eqref{chain}:
recalling also 
\eqref{mass} and \eqref{est0} and adding $ |m(\b u^\delta(t))|^2$ to both sides,
we easily obtain 
\begin{align*}
  &|m(\b u^\delta(t))|^2+
  \frac12\|\b u^\delta(t)-m(\b u^\delta(t))\|_*^2
  +\int_{Q_t}|\nabla u^\delta|^2
  +\delta\int_{\Sigma_t}|\nabla_\Gamma u_\Gamma^\delta|^2\\
  \nonumber
  &\qquad+C_{MZ}\left(\int_{Q_t}|\xi^\delta| + \int_{\Sigma_t}|\xi_\Gamma^\delta|\right)
  \\
  &\leq C + \frac12\|u_0^\delta - m (\b u_0^\delta)\|_*^2
  + \int_0^t (\b f^\delta -\b \pi (\b u^\delta),\b u^\delta - m(\b u_0^\delta)\b 1)_{\bsH}\, .
\end{align*}
Thanks to the Poincar\'e inequality \eqref{poin2},
the H\"older and Young inequalities together with the Lipschitz 
continuity of $\b\pi$ we infer that 
\begin{align*}
  &\|\b u^\delta(t)\|_*^2
  +\int_0^t\|u^\delta(s)\|_V^2\, ds
  +\delta\int_{\Sigma_t}|\nabla_\Gamma u_\Gamma^\delta|^2\\
  &\leq C
  \left(1 + \|\b f^\delta\|^2_{L^2(0,T;\bsH)} + 
  \int_0^t\|\b u^\delta(s)\|^2_{\bsH}\,ds\right)
  \qquad\forall\,t\in[0,T].
\end{align*}
At this point, since $\b Z\embed \b H$ with compact embedding, 
the following Ehrling lemma holds: 
\beq\label{ehrling}
  \forall\,\eps>0,\quad\exists\,C_\eps>0:\quad
  \|\b z\|_{\bsH}^2 \leq \eps \|\b z\|_{\bsZ}^2 + C_\eps\|\b z\|_{\bsV^*}^2
  \quad\forall\,\b z\in\b Z.
\eeq
Noting that by the trace theorems it holds 
$\|\b z\|_{\bsZ}\leq C\|z\|_V$ for every 
$\b z=(z,z_\Gamma)\in\b Z$, 
applying this inequality on the right-hand side and
taking also assumption \eqref{ip2} into account we infer that, for every $\eps>0$,
\begin{align*}
  &\|\b u^\delta(t)\|_*^2
  +\int_0^t\|u^\delta(s)\|_V^2\, ds
  +\delta\int_{\Sigma_t}|\nabla_\Gamma u_\Gamma^\delta|^2\\
  &\leq C+ \eps   \int_0^t\|u^\delta(s)\|_V^2\, ds + 
  C_\eps\int_0^t\|\b u^\delta(s)\|^2_*\,ds
  \qquad\forall\,t\in[0,T].
\end{align*}
Choosing for example $\eps=1/2$ and rearranging the terms, 
an application of the Gronwall lemma yields 
\beq
  \label{est_prelim}
  \|\b u^\delta\|_{L^\infty(0,T;\bsV^*)\cap L^2(0,T;\bsZ)}\leq C.
\eeq
}

\noindent{\bf Third estimate.} 
We proceed now in a formal but perhaps more explicative way,
referring to \cite{CF15} for a rigorous approach.
Testing 
\eqref{var1_delta_pos} by $\b\mu^\delta$, 
\eqref{var2_delta_pos} by $-\partial_t \b u^\delta$, summing, and integrating,
we obtain the formal energy inequality
\begin{align*}
  &\int_{Q_t}|\nabla\mu^\delta|^2 
  + \int_{\Sigma_t}|\nabla_\Gamma\mu_\Gamma^\delta|^2
  +\frac12\int_\Omega|\nabla u^\delta(t)|^2
  +\frac\delta2\int_\Gamma|\nabla_\Gamma u_\Gamma^\delta(t)|^2\\
  &\qquad+\int_\Omega(\widehat\beta+\widehat\pi)(u^\delta(t))
  +\int_\Gamma(\widehat\beta_\Gamma+\widehat\pi_\Gamma)(u_\Gamma^\delta(t))\\
  &\leq\frac12\int_\Omega|\nabla u^\delta_0|^2
  +\frac\delta2\int_\Gamma|\nabla_\Gamma u_{0\Gamma}^\delta|^2
  +\int_\Omega(\widehat\beta+\widehat\pi)(u^\delta_0)
  +\int_\Gamma(\widehat\beta_\Gamma+\widehat\pi_\Gamma)(u_{0\Gamma}^\delta)
  +\int_0^t(\partial_t\b u^\delta,\b f^\delta)_{\bsH}.
\end{align*}
Let us stress that this inequality is only formal, 
since that the last term is not well-defined 
in general for the regularities of $\partial_t\b u^\delta$ and $\b f^\delta$.
Nonetheless, we can give rigorous sense to it by exploiting the 
representation $\b f^\delta=\b g^\delta + \b h^\delta$ and 
using integration by parts in time: indeed, we formally have 
\begin{align*}
\int_0^t(\partial_t\b u^\delta,\b f^\delta)_{\bsH} =
(\b u^\delta(t),\b g^\delta(t))_{\bsH} - (\b u^\delta_0,\b g^\delta(0))_{\bsH} 
-\int_0^t(\takeshi{\b u^\delta},\partial_t \b g^\delta)_{\bsH} +
\int_0^t\langle\partial_t\b u^\delta,\b h^\delta\rangle_{\bsV^*, \bsV}.
\end{align*}
It is then clear that all terms above make sense, 
and a classical argument based on suitable approximations
of the problem at $\delta>0$ fixed (see \cite{CF15}) yields the rigorous estimate 
\begin{align*}
  &\int_{Q_t}|\nabla\mu^\delta|^2 
  + \int_{\Sigma_t}|\nabla_\Gamma\mu_\Gamma^\delta|^2
  +\frac12\int_\Omega|\nabla u^\delta(t)|^2
  +\frac\delta2\int_\Gamma|\nabla_\Gamma u_\Gamma^\delta(t)|^2
  \\
  &\qquad 
  \pier{+\int_\Omega\widehat\beta(u^\delta(t))
  +\int_\Gamma\widehat\beta_\Gamma(u_\Gamma^\delta(t))}\\
  &\leq\frac12\int_\Omega|\nabla u^\delta_0|^2
  +\frac\delta2\int_\Gamma|\nabla_\Gamma u_{0\Gamma}^\delta|^2
  +\int_\Omega(\widehat\beta+\widehat\pi)(u^\delta_0)
  +\int_\Gamma(\widehat\beta_\Gamma+\widehat\pi_\Gamma)(u_{0\Gamma}^\delta)\\
  &\qquad
\pier{- \int_\Omega\widehat\pi(u^\delta(t))
   - \int_\Gamma \widehat\pi_\Gamma(u_\Gamma^\delta(t))}  
  \\
  &\qquad 
  +(\b u^\delta(t),\b g^\delta(t))_{\bsH} - (\b u^\delta_0,\b g^\delta(0))_{\bsH} 
-\int_0^t(\takeshi{\b u^\delta},\partial_t \b g^\delta)_{\bsH} +
  \int_0^t\langle\partial_t\b u^\delta,\b h^\delta\rangle_{\bsV^*, \bsV}.
\end{align*}
Summing now the estimate \eqref{est0}
and using the fact that \eqref{norm2}
yields an equivalent norm in $V$, we obtain a 
control on the $V$-norm of $u^\delta(t)$ on the left-hand side.
\pier{Observe also that $\widehat\beta,\, \widehat\beta_\Gamma$ are nonnegative 
and that $\widehat\pi,\, \widehat\pi_\Gamma$ are at most with quadratic growth since 
$\pi,\, \pi_\Gamma$ are Lipschitz continuous. 
Consequently, by virtue of the bounds \eqref{ip1}--\eqref{ip2} on the data
and the Young inequality} we obtain that
\begin{align*}
  &\int_{Q_t}|\nabla\mu^\delta|^2 
  + \int_{\Sigma_t}|\nabla_\Gamma\mu_\Gamma^\delta|^2
  +\|u^\delta(t)\|_V^2
  +\delta\int_\Gamma|\nabla_\Gamma u_\Gamma^\delta(t)|^2\\
  &\leq
   C\left(1 +\pier{ \|\b u^\delta(t)\|_{\bsH}^2 } +
    \int_0^t\|\partial_t\b g^\delta(s)\|_{\bsH}
  \|\b u^\delta(s)\|_{\bsH}\,ds \right)+ 
  \frac14\|\partial_t \b u^\delta\|_{L^2(0,t; \bsV^*)}^2
\end{align*}
for a certain constant $C>0$ independent of $\delta$.
\pier{Now, recalling again that
the classical trace theory implies 
that $\|\b z\|_{\bsZ}\leq C\|z\|_V$ for every 
$\b z=(z,z_\Gamma)\in\b Z$, by the Ehrling inequality \eqref{ehrling} 
and the already proved estimate \eqref{est_prelim} we have, for all $\eps>0$,
\[
  \|\b u^\delta(t)\|_{\bsH}^2 \leq \eps\|u^\delta(t)\|_V^2 + C_\eps
  \qquad\forall\,t\in[0,T],
\]
where $C_\eps$ is independent of $t$ and $\delta$. Hence, 
choosing $\eps$ small enough and rearranging the terms, 
in view also of the 
inequality \eqref{est00}
on the right-hand side, we infer that
}
\begin{align*}
  &\int_{Q_t}|\nabla\mu^\delta|^2 
  + \int_{\Sigma_t}|\nabla_\Gamma\mu_\Gamma^\delta|^2
  +\pier{\frac12}\|u^\delta(t)\|_V^2
  +\delta\int_\Gamma|\nabla_\Gamma u_\Gamma^\delta(t)|^2\\
  &\leq C\left(1 + \int_0^t\|\partial_t\b g^\delta(s)\|_{\bsH}
  \|u^\delta(s)\|_{V}\,ds \right)+ 
  \frac12\int_{Q_t}|\nabla\mu^\delta|^2 
  + \frac12\int_{\Sigma_t}|\nabla_\Gamma\mu_\Gamma^\delta|^2.
\end{align*}
Hence, rearranging the terms yields, thanks to the Gronwall lemma, that
\begin{align}
  \label{est1}
  \|\b u^\delta\|_{L^\infty(0,T;\bsZ)} + \delta^{1/2}\|\b u^\delta\|_{L^\infty(0,T;\bsV)} &\leq C,\\
  \label{est2}
  \|\nabla\mu^\delta\|_{L^2(0,T; H)} + 
  \|\nabla_\Gamma\mu_\Gamma^\delta\|_{L^2(0,T; H_\Gamma)} &\leq C.
\end{align}
Moreover, the estimate \eqref{est00} implies also that
\beq
  \label{est4}
  \|\partial_t \b u^\delta\|_{L^2(0,T; \bsV^*)}\leq C.
\eeq

\noindent{\bf Fourth estimate.} 
By comparison in \eqref{est_aux} we have, almost everywhere on $(0,T)$,
\begin{align*}
   \|\xi^\delta\|_{L^1(\Omega)} +
   \|\xi_\Gamma^\delta\|_{L^1(\Gamma)}
  &\leq 
  C\left(1+\|\partial_t\b u^\delta\|_{\bsV^*}+
  \|\b f^\delta\|_{\bsH}\right)
  \left(1+
  \|\b u^\delta\|_{L^\infty(0,T; \bsH)}\right),
\end{align*}
and the estimates \eqref{est1}--\eqref{est4} yield
\beq
  \label{est5}
  \|\xi^\delta\|_{L^2(0,T; L^1(\Omega))} + \|\xi_\Gamma^\delta\|_{L^2(0,T; L^1(\Gamma))} \leq C.
\eeq
Testing \eqref{var2_delta_pos} by \takeshi{$\b 1/(|\Omega|+|\Gamma|)$}
and using the estimates \eqref{est1}, \eqref{est5}, together with the 
Lipschitz continuity of $\b \pi$,
it follows that 
\beq
  \label{est6}
  \|m(\b\mu^\delta)\|_{L^2(0,T)} \leq C,
\eeq
so that by \eqref{est2} and the 
equivalent norm \eqref{norm1} in $\bV$ we get 
\beq
  \label{est7}
  \|\b\mu^\delta\|_{L^2(0,T;\bsV)}\leq C.
\eeq

\noindent{\bf Fifth estimate.} 
The idea now is to test 
equation \eqref{var2_delta_pos} by $(\xi^\delta, \xi^\delta_{|\Gamma})$:
however, this is only formal due to the regularity of $\xi^\delta$.
To make it rigorous, we recall that from \cite{CF15}
the system \eqref{eq1}--\eqref{eq2_init}
can be seen as limit as $\lambda\searrow0$
of a suitable approximated system
where $\beta$ and $\beta_{\Gamma}$ are replaced by 
their Yosida approximations $\beta_\lambda$ and $\beta_{\Gamma,\lambda}$,
with $\lambda\in(0,\lambda_0)$.
In this case, $\xi^\delta$ is exactly the weak limit in $L^2(0,T; H)$
of the respective sequence $\beta_\lambda(u^\delta_\lambda)$.
At this level, the Lipschitz-continuity of $\beta_\lambda$
yields the desired regularity $(\beta_\lambda(u_\lambda^\delta), 
\beta_\lambda(u_{\lambda,\Gamma}^\delta))\in \bV$ almost everywhere in $(0,T)$.
Hence, testing the $\lambda$-regularised equation \eqref{var2_delta_pos} by
$(\beta_\lambda(u_\lambda^\delta), 
\beta_\lambda(u_{\lambda,\Gamma}^\delta))$ yields 
\begin{align*}
  &\int_{Q_t}\beta_\lambda'(u_\lambda^\delta)|\nabla u_\lambda^\delta|^2
  +\delta\int_{\Sigma_t}\beta_{\lambda}'(u_{\lambda,\Gamma}^\delta)
  |\nabla_\Gamma u_{\lambda,\Gamma}^\delta|^2
  +\int_{Q_t}|\beta_\lambda(u_\lambda^\delta)|^2
  +\int_{\Sigma_t}\beta_\lambda(u_{\lambda,\Gamma}^\delta)
  \beta_{\Gamma,\lambda}(u_{\lambda,\Gamma}^\delta)\\
  &\qquad
  =\int_0^t(\takeshi{ \mu^\delta_\lambda -
  f^\delta - \pi(u_\lambda^\delta)},\beta_\lambda(u_\lambda^\delta))_H
  +\int_0^t(\mu^\delta_{\lambda,\Gamma}-
  f_\Gamma^\delta - \pi_\Gamma(u_{\lambda,\Gamma}^\delta),
  \beta_\lambda(u_{\lambda,\Gamma}^\delta))_{H_\Gamma}.
\end{align*}
Now, we note that \eqref{dom_beta} yields an analogous 
inequality on the Yosida approximations $\beta_\lambda$ and 
$\beta_{\Gamma,\lambda}$ (see for example \cite{CC13}), from which we 
have the control from below 
\[
  \int_{\Sigma_t}\beta_\lambda(u_{\lambda,\Gamma}^\delta)
  \beta_{\Gamma,\lambda}(u_{\lambda,\Gamma}^\delta)\geq
  \frac1{2M}\int_{\Sigma_t}|\beta_\lambda(u_{\lambda,\Gamma}^\delta)|^2 - C.
\]
Consequently,
by the monotonicity of $\beta_\lambda$,
the Young inequality, the Lipschitz continuity of $\b\pi$,
and the estimate \eqref{est5}
we obtain, after rearranging the terms,
\[
  \|\beta_\lambda(u^\delta_\lambda)\|_{L^2(0,T; H)}^2
  +\|\beta_{\lambda}(u^\delta_{\lambda,\Gamma})\|_{L^2(0,T; H_\Gamma)}^2
  \leq C\left(1+ \|\b u^\delta_\lambda\|_{L^2(0,T;\bsH)}^2\right),
\]
where $C>0$ is independent of both $\delta$ and $\lambda$.
Consequently, by the estimate \eqref{est1} we obtain
\beq
  \label{est8'}
  \|\beta_\lambda(u_{\lambda}^\delta)\|_{L^2(0,T; H)}+
  \|\beta_\lambda(u_{\lambda,\Gamma}^\delta)\|_{L^2(0,T; H_\Gamma)}\leq C
  \qquad\forall\,\lambda\in(0,\lambda_0),
\eeq
from which it follows in particular, 
by weak lower semicontinuity as $\lambda\searrow0$, that
\beq
  \label{est8}
  \|\xi^\delta\|_{L^2(0,T; H)} \leq C.
\eeq
Now, in view of Remark~\ref{rem1},
by comparison in equation \eqref{pier1} and the estimates just proved we have 
\beq
  \label{est9}
  \|\Delta u^\delta\|_{L^2(0,T; H)}\leq C.
\eeq
By the classical trace theorems \cite[Thm.~2.27]{BG87}
and elliptic regularity \cite[Thm.~3.2]{BG87},
the estimates \eqref{est1} and \eqref{est9} yield
\beq
  \label{est10}
  \|\partial_\nu u^\delta\|_{L^2(0,T; Z_\Gamma^*)}
  +\delta^{1/2}\|\partial_\nu u^\delta\|_{L^2(0,T; H_\Gamma)}\leq C,
\eeq
so that by comparison in \eqref{pier2}
and estimate \eqref{est7} we infer that
\beq
  \label{est11}
  \|-\delta\Delta_\Gamma u_\Gamma^\delta 
  + \xi_\Gamma^\delta\|_{L^2(0,T; Z_\Gamma^*)}\leq C.
\eeq
Eventually, this implies together with \eqref{est1} that 
\beq
  \label{est12}
  \delta^{1/2}\|\Delta_\Gamma u_\Gamma^\delta\|_{L^2(0,T; V_\Gamma^*)}+
  \|\xi_\Gamma^\delta\|_{L^2(0,T; V_\Gamma^*)}\leq C.
\eeq

\subsection{Passage to the limit}
From the estimates \eqref{est1}--\eqref{est12}
and weak and weak* compactness, we infer that 
there exists a triplet $(\b u, \b\mu, \b\xi)$ with
\begin{align*}
 &\b u\in H^1(0,T; \bV^*)\cap L^\infty(0,T;\bZ), \qquad\Delta u\in L^2(0,T; H),\\
 &\b\mu\in L^2(0,T; \bV),\\
 &\b\xi\in L^2(0,T; H\times Z_\Gamma^*),
\end{align*}
such that, as $\delta\searrow0$, on a possibly relabelled subsequence, 
\begin{align}
  \label{conv1}
  \b u^\delta \wstarto \b u \quad&\text{in } H^1(0,T; \bV^*)\cap L^\infty(0,T;\bZ),\\
  \Delta u^\delta \wto \Delta u \quad&\text{in } L^2(0,T; H),\\
  \b \mu^\delta \wto \b \mu \quad&\text{in }  L^2(0,T;\bV),\\
  \label{conv4}
  \b\xi^\delta \wto \b\xi \quad&\text{in }  L^2(0,T; H\times V_\Gamma^*),\\
  \delta \b u^\delta \to \takeshi{\b 0} \quad&\text{in } L^\infty(0,T; \bV),\\
  \label{conv6} 
  -\delta\Delta_\Gamma u_\Gamma^\delta + 
  \xi^\delta_{{\takeshi \Gamma}} \wto  \xi_\Gamma \quad&\text{in } L^2(0,T; Z_\Gamma^*).
\end{align}
In particular, 
by the Aubin\takeshi{--}Lions and Simon compactness results 
(see e.g.~\cite[\S~8, Cor.~4]{Sim87}),
the compact inclusion $\bZ\embed \bH$
implies that
\beq
  \label{conv7}
  \b u^\delta \to \b u \quad\text{in } C^0([0,T]; \bH),
\eeq
which can be rewritten as
\[
  u^\delta \to u \quad\text{in } C^0([0,T]; H), \qquad
  u^\delta_\Gamma \to u_\Gamma \quad\text{in } C^0([0,T]; H_\Gamma).
\]
Hence, the Lipschitz continuity of $\pi$ and $\pi_\Gamma$ implies also that 
\[
  \pi(u^\delta) \to \pi(u) \quad\text{in } C^0([0,T]; H), \qquad
  \pi_\Gamma(u^\delta_\Gamma) \to \pi_\Gamma(u_\Gamma)
   \quad\text{in } C^0([0,T]; H_\Gamma),
\]
and therefore 
\beq
  \label{conv8}
\b\pi(\b u^\delta)\to\b\pi(\b u) \qquad\text{in } C^0([0,T]; \bH).
\eeq
Hence, passing to the weak limit in \eqref{var1_delta_pos}--\eqref{var2_delta_pos} yields exactly \eqref{var1_delta0}--\eqref{var2_delta0}.

In order to conclude, we only need to prove conditions \eqref{incl1_delta0}--\eqref{incl2_delta0}.
To this end, the \takeshi{demi-closedness} 
of the maximal monotone operator 
$\beta$ yields $\xi\in\beta(u)$ almost everywhere in $Q$
by the classical results in \cite{Bar10, Bre73}.
Moreover, testing \eqref{var2_delta_pos} by $\b u^\delta$ gives
\begin{align*}
  &\int_Q|\nabla u^\delta|^2 + \delta\int_\Sigma|\nabla_\Gamma u_\Gamma^\delta|^2
  +\int_Q\xi^\delta u^\delta + \int_\Sigma\xi_\Gamma^\delta u_\Gamma^\delta\\
  &\qquad=\int_Q(\mu^\delta + f^\delta - \pi(u^\delta))u^\delta
  +\int_\Sigma(\mu_\Gamma^\delta + f_\Gamma^\delta - \pi_\Gamma(u_\Gamma^\delta))
  u_\Gamma^\delta,
\end{align*}
while testing \eqref{var2_delta0} by $\b u$ gives
\begin{align*}
  &\int_Q|\nabla u|^2
  +\int_Q\xi u + \int_0^T\langle\xi_\Gamma, u_\Gamma\rangle_{Z_\Gamma^*, Z_\Gamma}\\
  &\qquad=\int_Q(\mu + f - \pi(u))u
  +\int_\Sigma(\mu_\Gamma + f_\Gamma - \pi_\Gamma(u_\Gamma))
  u_\Gamma.
\end{align*}
By lower semicontinuity and weak-strong convergence we deduce that
\begin{align*}
  \limsup_{\delta\searrow0}\int_\Sigma\xi_\Gamma^\delta u_\Gamma^\delta
  &\leq\limsup_{\delta\searrow0}\int_Q(\mu^\delta + f^\delta - \pi(u^\delta))u^\delta
  +\limsup_{\delta\searrow0}
  \int_\Sigma(\mu_\Gamma^\delta + f_\Gamma^\delta - \pi_\Gamma(u_\Gamma^\delta))
  u_\Gamma^\delta\\
  &-\liminf_{\delta\searrow0}\int_Q|\nabla u^\delta|^2
  -\liminf_{\delta\searrow0}\int_Q\xi^\delta u^\delta\\
  &\leq \int_Q(\mu + f - \pi(u))u
  +\int_\Sigma(\mu_\Gamma + f_\Gamma - \pi_\Gamma(u_\Gamma))
  u_\Gamma - \int_Q|\nabla u|^2 - \int_Q\xi u,
\end{align*}
from which 
\beq
  \label{limsup}
  \limsup_{\delta\searrow0}\int_\Sigma\xi_\Gamma^\delta u_\Gamma^\delta
  \leq \int_0^T\langle\xi_\Gamma, u_\Gamma\rangle_{Z_\Gamma^*, Z_\Gamma}.
\eeq
Now, condition \eqref{incl2_delta_pos} and the 
subdifferential property
$\beta_\Gamma=\partial\widehat\beta_\Gamma$ yields 
\[
  \int_\Sigma\widehat\beta_\Gamma(u_\Gamma^\delta)
  +\int_\Sigma\xi_\Gamma^\delta(z_\Gamma-u_\Gamma^\delta)
  \leq \int_\Sigma\widehat\beta_\Gamma(z_\Gamma) 
  \qquad\forall\,z_\Gamma\in L^2(0,T; H_\Gamma).
\]
Choosing now $z_\Gamma\in L^2(0,T; V_\Gamma)$,
using the convergences \eqref{conv4} and \eqref{conv7},
the weak lower semicontinuity of $\widehat\beta_\Gamma$,
and \eqref{limsup}, we have
\begin{align*}
  &\int_\Sigma\widehat\beta_\Gamma(u_\Gamma)\leq 
  \liminf_{\delta\searrow0}\int_\Sigma\widehat\beta_\Gamma(u_\Gamma^\delta),\\
  &\int_0^T\langle\xi_\Gamma,z_\Gamma\rangle_{Z_\Gamma^*,Z_\Gamma}
  =\lim_{\delta\searrow0}\int_\Sigma\xi_\Gamma^\delta z_\Gamma,\\
  &-\int_0^T\langle\xi_\Gamma,u_\Gamma\rangle_{Z_\Gamma^*,Z_\Gamma}
  \leq -\limsup_{\delta\searrow0}
  \int_\Sigma\xi_\Gamma^\delta u_\Gamma^\delta
  =\liminf_{\delta\searrow0} 
  \left(- \int_\Sigma\xi_\Gamma^\delta u_\Gamma^\delta\right).
\end{align*}
Hence, passing to the $\liminf$ as $\delta\searrow0$ we obtain
\[
  \int_\Sigma\widehat\beta_\Gamma(u_\Gamma)
  +\int_0^T\langle\xi_\Gamma,z_\Gamma-u_\Gamma\rangle_{Z_\Gamma^*,Z_\Gamma}
  \leq \int_\Sigma\widehat\beta_\Gamma(z_\Gamma) 
  \qquad\forall\,z_\Gamma\in L^2(0,T; V_\Gamma).
\]
We infer now that such inequality holds also for all $z_\Gamma\in L^2(0,T; Z_\Gamma)$.
Indeed, given an arbitrary $z_\Gamma\in L^2(0,T; Z_\Gamma)$, 
for $\eps>0$ we can set
$z_\Gamma^\eps\in L^2(0,T; W_\Gamma)$ as the unique solution to the elliptic problem 
\[
  z_\Gamma^\eps - \eps\Delta_\Gamma z_\Gamma^\eps = z_\Gamma
  \quad\text{on } \Sigma.
\]
Then, it is not difficult to show by standard testing techniques \takeshi{(see\pier{, e.g.,}\ \cite[Lemma A.1]{CF16})} that 
\[
z_\Gamma^\eps\to z_\Gamma \quad\text{in } L^2(0,T;Z_\Gamma), \qquad
\widehat\beta_\Gamma(z_\Gamma^\eps)\leq \widehat\beta_\Gamma(z_\Gamma)
\quad\text{a.e.~on } \Sigma,
\]
so that letting $\eps\to0$ in the subdifferential relation we can conclude.
This shows that $(\b u, \b \mu, \b \xi)$ is a weak solution
to the system with $\delta=0$ in the sense of Definition~\ref{sol:delta0}.

\subsection{Continuous dependence}
Let $\{(\b u_i, \b\mu_i, \b\xi_i)\}_{i=1,2}$
be two weak solutions of the system \eqref{eq1}--\eqref{eq2_init}
in the sense of Definition~\ref{sol:delta0},
with respect to the data $\{(\b u_{0,i}, \b f_i)\}_{i=1,2}$. 
Then, setting $\takeshi{\bar{\b u}}:=\b u_1-\b u_2$, $\takeshi{\bar{\b \mu}}:=\b\mu_1-\b\mu_2$,
$\takeshi{\bar{\b \xi}}:=\b\xi_1-\b\xi_2$, $\takeshi{\bar{\b u}_0}:=\b u_{0,1}-\b u_{0,2}$, and 
$\takeshi{\bar{\b f}}:=\b f_1-\b f_2$, it holds that 
\begin{align}
  \label{diff1}
  \langle\partial_t \takeshi{\bar{\b u}},\b z\rangle_{\bsV^*,\bsV}
  +\int_\Omega\nabla \fukao{\bar \mu}\cdot\nabla z
  +\int_\Gamma\nabla_\Gamma \fukao{\bar \mu}_\Gamma\cdot\nabla_\Gamma z_\Gamma=0 \nonumber \\
  \hbox{for every $\b z\in\bV$, \, a.e. in $(0,T)$,}\\
  \label{diff2}
  (\takeshi{\bar{\b \mu}},\b z)_{\bsH}=
  \int_\Omega\nabla \fukao{\bar u}\cdot\nabla z
  +(\fukao{\bar \xi},z)_H+\langle \fukao{\bar \xi}_\Gamma,z_\Gamma \rangle_{Z_\Gamma^*,Z_\Gamma}
  +(\b \pi(\b u_1)- \takeshi{\b \pi}(\b u_2) - \takeshi{\bar{\b f}}, \b z)_{\bsH}\nonumber \\
  \hbox{for every $\b z\in\bZ$, \, a.e. in $(0,T)$.}
  \end{align}
Recalling that $m(\b u_{0,1})=m(\b u_{0,2})$, 
testing \eqref{diff1} by \takeshi{$\b 1/(|\Omega| + |\Gamma|)$} it follows that 
\beq\label{zero_m}
  m(\takeshi{\bar{\b u}}(t))=0 \quad\forall\,t\in[0,T].
\eeq
Consequently, we can test \eqref{diff1} by $\mathcal L^{-1} \takeshi{\bar{\b u}}$,
\eqref{diff2} by $-\takeshi{\bar{\b u}}$, integrate in time, and add the respective equations.
Noting that there is a cancellation (as pointed out in \eqref{est_aux}),  thanks 
to the chain rule \eqref{chain} we obtain that
\begin{align*}
  \frac12\norm{ \takeshi{\bar{\b u}}(t)}_{*}^2 + 
  \int_{Q_t}|\nabla \takeshi{\bar u}|^2 + \int_{Q_t} \takeshi{\bar \xi} \takeshi{\bar u} + 
  \int_0^t\langle \takeshi{\bar \xi}_\Gamma, \takeshi{\bar u}_\Gamma\rangle
  =\frac12\norm{\takeshi{\bar{\b u}}_0}_{*}^2 + 
  \int_0^t(\takeshi{\bar{\b f}} + \b\pi(\b u_2)-\b\pi(\b u_1), \takeshi{\bar{\b u}})_{\bsH}.
\end{align*}
Exploiting condition \eqref{zero_m},
the monotonicity of $\beta$ and 
$\beta_\Gamma$, the Lipschitz continuity of $\b\pi$, 
and the Young inequality,
we infer that 
\begin{align*}
  \norm{\takeshi{\bar{\b u}}(t)}_{\bsV^*}^2 + 
  \int_0^t\norm{\takeshi{\bar{\b u}}}_{\bsZ}^2 
  \leq C\left(\norm{\takeshi{\bar{\b u}}_0}_{\bsV^*}^2 + 
  \int_0^t\norm{\takeshi{\bar{\b f}} }_{\bsH}^2 + \int_0^t\norm{\takeshi{\bar{\b u}}}_{\bsH}^2\right).
\end{align*}
At this point, applying the Ehrling inequality \eqref{ehrling} on the right-hand side, 
choosing $\eps>0$ sufficiently small, and rearranging the terms, 
we deduce, possibly renominating the constant $C$, that 
 \begin{align*}
  \norm{\takeshi{\bar{\b u}}(t)}_{\bsV^*}^2 + 
  \frac12\int_0^t\norm{\takeshi{\bar{\b u}}}_{\bsZ}^2 
  \leq C\left(\norm{\takeshi{\bar{\b u}}_0}_{\bsV^*}^2 + 
  \int_0^t\norm{\takeshi{\bar{\b f}} }_{\bsH}^2 + \int_0^t\norm{\takeshi{\bar{\b u}} }_{\bsV^*}^2\right).
\end{align*}
Then, the conclusion follows by applying the Gronwall lemma.
The proof of Theorem~\ref{thm:conv} is thus complete.

\subsection{Refined convergence}
Here we prove Theorem~\ref{thm:conv2}.
We show that the extra assumption \eqref{ip_extra} on the graphs
yields additional estimates on the solutions.

First of all, 
since assumption \eqref{ip_extra} induces the analogous 
inequalities on the respective  Yosida approximations
(details are given in \cite[Appendix]{CF20}),
the estimate \eqref{est8'} implies
\[
  \|\beta_\lambda(u_{\lambda}^\delta)\|_{L^2(0,T; H)}+
  \|\beta_{\Gamma,\lambda}(u_{\lambda,\Gamma}^\delta)\|_{L^2(0,T; H_\Gamma)}\leq C
  \qquad\forall\,\lambda\in(0,\lambda_0),
\]
from which, taking the limit as $\lambda\searrow0$,
\beq
  \label{est13}
  \|\xi^\delta\|_{L^2(0,T; H)}+
  \|\xi^\delta_\Gamma\|_{L^2(0,T; H_\Gamma)}\leq C.
\eeq
Now, recalling Remark~\ref{rem1}, by comparison
in \eqref{pier2} and using 
the estimate \eqref{est10}  we have
\beq\label{est14}
  \| \partial_\nu u^\delta - \delta\Delta_\Gamma u_\Gamma^\delta\|_{L^2(0,T; H_\Gamma)}
  +\delta\|\Delta_\Gamma u_\Gamma^\delta\|_{L^2(0,T; Z_\Gamma^*)} \leq C.
\eeq
At this point, one can pass to the limit as $\delta\searrow0$ as above
by exploiting the additional estimates  \eqref{est13}--\eqref{est14},
which yield the extra regularities
\[
  \xi_\Gamma \in L^2(0,T; H_\Gamma), \qquad
  \partial_\nu u \in L^2(0,T; H_\Gamma).
\]
By elliptic regularity (see \cite[Thm.~3.2]{BG87}) this implies that 
\[  
  u  \in L^2(0,T; H^{3/2}(\Omega)),
\]
hence also by the trace theory that 
\[
  u_\Gamma \in L^2(0,T; V_\Gamma).
\]
Eventually, the pointwise inclusion $\xi_\Gamma\in\beta_\Gamma(u_\Gamma)$
almost everywhere on $\Sigma$ can be obtained arguing as in 
Remark~\ref{rem3}. This concludes the proof of Theorem~\ref{thm:conv2}.

\subsection{Error estimate}
Here we prove Theorem~\ref{thm:conv3}.
To this end, taking the difference of the variational formulations
\eqref{var1_delta_pos}--\eqref{var2_delta_pos}
and \eqref{var1_delta0}--\eqref{var2_delta0}, we obtain, thanks to
the additional regularity of $\xi_\Gamma$, that
\begin{align}
  \label{diff1_err}
  \langle\partial_t(\b u^\delta-\b u),\b z\rangle_{\bsV^*,\bsV}
  +\int_\Omega\nabla(\mu^\delta-\mu)\cdot\nabla z
  +\int_\Gamma\nabla_\Gamma(\mu^\delta_\Gamma-\mu_\Gamma)
  \cdot\nabla_\Gamma z_\Gamma=0
\end{align}
and
\begin{align}
  \nonumber
  (\b\mu^\delta - \b\mu,\b z)_{\bsH}&=
  \int_\Omega\nabla (u^\delta-u)\cdot\nabla z
  +\delta\int_\Gamma\nabla_\Gamma u_\Gamma^\delta
  \cdot\nabla z_\Gamma\\
  \label{diff2_err}
  &+(\b\xi^\delta - \b\xi + \b\pi(\b u^\delta)-\takeshi{\b \pi}(\b u) +\b f - \b f^\delta, \b z)_{\bsH}
\end{align}
for every $\b z\in\bV$, almost everywhere in $(0,T)$.
Now, since $m(\b u_0^\delta)=m(\b u_0)$ by assumption, 
testing equation \eqref{diff1_err} by \takeshi{$\b 1/(|\Omega| + |\Gamma|)$} we infer that 
\beq\label{zero_m_err}
  m((\b u^\delta - \b u)(t))=0 \quad\forall\,t\in[0,T].
\eeq
Hence, one can test \eqref{diff1_err} by $\mathcal L^{-1}(\b u^\delta-\b u)$,
\eqref{diff2_err} by $-(\b u^\delta-\b u)$, 
integrate in time, and add the respective equations:
taking into account the usual cancellation of terms and 
\eqref{chain}, we obtain 
\begin{align*}
  &\frac12\|(\b u^\delta- \b u)(t)\|_{*}^2 + 
  \int_{Q_t}|\nabla (u^\delta - u)|^2 + \delta\int_{\Sigma_t}
  |\nabla_\Gamma u^\delta_\Gamma|^2+
  \int_0^t(\b\xi^\delta-\b\xi,\b u^\delta-\b u)_{\bsH} \\
  &=\frac12\|\b u_0^\delta - \b u_0\|_{*}^2 
  +\delta\int_{\Sigma_t}\nabla_\Gamma u^\delta_\Gamma\cdot
  \nabla_\Gamma u_\Gamma+ 
  \int_0^t(\b f^\delta - \b f + \b\pi(\b u)-\b\pi(\b u^\delta), \b u^\delta - \b u)_{\bsH}.
\end{align*}
At this point, taking condition \eqref{zero_m} into account
on the left-hand side together with
the monotonicity of $\beta$ and 
$\beta_\Gamma$, and using
the Lipschitz continuity of $\b\pi$, 
and the Young inequality on the right-hand side,
we infer that 
\begin{align*}
  &\|(\b u^\delta - \b u)(t)\|_{\bsV^*}^2 + 
  \int_0^t\|\b u^\delta - \b u\|_{\bsZ}^2 
  +\delta\int_0^t\|\nabla_\Gamma u_\Gamma^\delta\|_{H_\Gamma}^2 \\
  &\leq C\left(\delta\int_{\Sigma_t}\nabla_\Gamma u^\delta_\Gamma\cdot
  \nabla_\Gamma u_\Gamma+
  \|\b u_0^\delta - \b u_0\|_{\bsV^*}^2 + 
  \int_0^t\|\b f^\delta - \b f\|_{\bsH}^2 
  + \int_0^t\|\b u^\delta - \b u\|_{\bsH}^2\right).
\end{align*}
Now, using the Young inequality and the regularity of $\b u$ one has
\[
  \delta\int_{\Sigma_t}\nabla_\Gamma u^\delta_\Gamma\cdot
  \nabla_\Gamma u_\Gamma \leq 
  \frac\delta2\int_0^t\|\nabla_\Gamma u_\Gamma^\delta\|_{H_\Gamma}^2+
  \frac\delta2\|\b u\|_{L^2(0,T; \bsV)}^2.
\]
Consequently,
using the Ehrling inequality \eqref{ehrling}
on the right-hand side and rearranging the terms we obtain,
updating the value of $C$,
\begin{align}
  \nonumber
  &\|(\b u^\delta-\b u)(t)\|_{\bsV^*}^2 + 
  \int_0^t\|\b u^\delta - \b u\|_{\bsZ}^2
  +\frac\delta2\int_0^t\|\nabla_\Gamma u_\Gamma^\delta\|_{H_\Gamma}^2 \\
  \label{error_aux}
  &\leq C\left(\delta+\|\b u_0^\delta - \b u_0\|_{\bsV^*}^2 + 
  \|\b f^\delta - \b f\|_{L^2(0,T;\bsH)}^2 + \int_0^t\|\b u^\delta - \b u\|_{\bsV^*}^2\right).
\end{align}
The Gronwall lemma yields the desired error estimate, hence also 
the rate of convergence. Moreover, 
we note that this implies also the boundedness of \takeshi{$\{u_\Gamma^\delta\}_{\delta \in (0,1)}$}
in $L^2(0,T; V_\Gamma)$, from which the weak convergence 
\[
  \b u^\delta \wto \b u \qquad\text{in } L^2(0,T;\b V)
\]
follows as $\delta\searrow0$.
 This concludes the proof of Theorem~\ref{thm:conv3}.

\appendix
\section{Appendix}
\setcounter{equation}{0}
\label{appendix}

\begin{lem}
\label{lem:A1}
  In the setting of Section~\ref{sec:main}, there exists a constant $C_p>0$ such that 
  \[
  \|z\|_V\leq C_p\|\nabla z\|_H \qquad\forall\, \b z=(z,z_\Gamma)\in \b Z_0.
  \] 
\end{lem}
\begin{proof}
It is enough to prove that there exists $C>0$ such that 
\[
  \|z\|_H\leq C\|\nabla z\|_H \qquad\forall\, \b z=(z,z_\Gamma)\in \b Z_0.
\] 
By contradiction, suppose that there exists a sequence \takeshi{$\{
\b z_n\}_{n \in \mathbb{N}} \subset\b Z_0$} 
such that 
\[
  \|z_n\|_H> n\|\nabla z_n\|_H \qquad\forall\, n\in\enne.
\] 
Then, setting $\b w_n:=\b z_n/\|z_n\|_H$, $n\in\enne$, it holds for every $n\in\enne$ that 
\[
  \|w_n\|_H=1, \qquad
  \|\nabla w_n\|_H<\frac1n\,, \qquad
  m(\b w_n)=0.
\]
We deduce that there exists $w\in V$ such that 
\[
  w_n\to w \quad\text{in } H, \qquad
  w_n\wto w \quad\text{in } V, \qquad
  \|w\|_H=1, \qquad
  \nabla w=0\,.
\]
In particular, setting $w_\Gamma:=w_{|\Gamma}$
it holds that $\b w:=(w,w_\Gamma)\in\b Z$. Since 
$w$ is constant with $\|w\|_H=1$, it necessarily holds that $m(\b w)\neq0$.
However, the weak convergence
$w_n\wto w$ in $V$ yields in particular that 
\[
  0=m(\b w_n)\to m(\b w),
\]
which is absurd. This completes the proof.
\end{proof}

\begin{prop}
  \label{prop:A2}
  Let $\b u_0$ satisfy \eqref{ip1_delta0}.
  Then, if \eqref{ip_extra} holds there exists a sequence 
  $\{\b u_0^\delta\}_{\delta\in(0,1)}$ satisfying \eqref{ip1} and
  such that $\b u_0^\delta \wto \b u$ in  $\b Z$ as $\delta\searrow0$.
\end{prop}
\begin{proof}
  In order to introduce a family $\{\b u_0^\delta\}_{\delta\in(0,1)}$
  we consider the elliptic system
  \begin{align}
  \label{cf1}
  u_0^\delta - \delta\Delta u_0^\delta = u_0 \qquad&\text{a.e.~in } \Omega,\\
  \label{cf2}
  u_{0|\Gamma}^\delta=u_{0,\Gamma}^\delta,\quad
  -\partial_\nu u_0^\delta \in
  \beta_\Gamma(u_{0,\Gamma}^\delta)
  \qquad&\text{a.e.~on } \Gamma.
  \end{align}
  Note that \eqref{cf1}--\eqref{cf2} admits a unique solution 
  $\b u_0^\delta=(u_0^\delta, u_{0,\Gamma}^\delta)$
  with $u_0^\delta \in W$, as proved, e.g., in\cite[Prop.~2.9, p.~62]{Bar10}.
  Now, setting $\xi_{0,\Gamma}^\delta:=-\partial_\nu u_0^\delta\in Z_\Gamma$
  we have that \takeshi{$\xi_{0,\Gamma}^\delta\in\beta_\Gamma(u_{0,\Gamma}^\delta)$}
  almost everywhere on $\Gamma$.
  So, testing \eqref{cf1} by $u_{0}^\delta-\Delta u_0^\delta$, we
  integrate by parts with the aid of
  the boundary conditions in \eqref{cf2}.
  Thanks to the Young inequality
  we obtain exactly 
  \begin{align*}
  &\frac12\|u_0^\delta\|_H^2 +\frac12\|\nabla u_0^\delta\|_H^2
  +(\xi_{0,\Gamma}^\delta, u_{0,\Gamma}^\delta-u_{0,\Gamma})_{H_\Gamma}\\
  &\qquad+\delta\|\nabla u_0^\delta\|_H^2 +
   \delta\|\Delta u_{0}^\delta\|_{H}^2
   +\delta(\xi_{0,\Gamma}^\delta, u_{0,\Gamma}^\delta)_{H_\Gamma}\\
   &\leq
   \frac12\|u_0\|_H^2 + \frac12\|\nabla u_0\|_H^2.
  \end{align*}
  At this point, recalling that $\beta_\Gamma=\partial\widehat\beta_\Gamma$,
  we have that 
  \[
  (\xi_{0,\Gamma}^\delta, u_{0,\Gamma}^\delta-u_{0,\Gamma})_{H_\Gamma}
  \geq \int_\Gamma\widehat\beta_\Gamma(u_{0,\Gamma}^\delta)
  -\int_\Gamma\widehat\beta_\Gamma(u_{0,\Gamma}),
  \]
  while by monotonicity of $\beta_\Gamma$ and the fact that $0\in\beta_\Gamma(0)$
  it holds that 
  \[
  \delta(\xi_{0,\Gamma}^\delta, u_{0,\Gamma}^\delta)_{H_\Gamma} \geq0.
  \]
  Consequently, we infer that 
  \beq\label{cf3}
  \frac12\|u_{0}^\delta\|_V^2 + 
  \|\widehat\beta_\Gamma(u_{0,\Gamma}^\delta)\|_{L^1(\Gamma)}
   +\delta\|\Delta u_{0}^\delta\|_{H}^2
   \leq \frac12\|u_{0}\|_V^2
   +\|\widehat\beta_\Gamma(u_{0,\Gamma})\|_{L^1(\Gamma)},
  \eeq
  where the right-hand side is finite due to \eqref{ip1_delta0}.
  This readily implies that there exists $v_0\in V$ such that,
  in principle along a subsequence,
  \[
  u_0^\delta \wto v_0 \quad\text{in } V, \qquad
  \delta\Delta u_0^\delta \to 0 \quad\text{in } H.
  \]
  Passing to the limit in \eqref{cf1} we realise that 
  $u_{0,\delta}\to u_0$ in $H$
  along the entire family $\delta\searrow0$,
  hence also that $u_0=v_0$ almost everywhere in $\Omega$.
  Moreover, we recall that the system \eqref{cf1}--\eqref{cf2} can be seen
  as the limit as $\lambda\searrow0$ of the corresponding one 
  where $\beta_\Gamma$ is replaced by its Yosida approximation 
  $\beta_{\Gamma, \lambda}$. Hence, 
  testing the respective equation 
  approximating \eqref{cf1} by $\beta_{\lambda}(u_0^{\delta,\lambda})$,
  where $\beta_\lambda$ is the Yosida approximation of $\beta$,
  we obtain
  \[
  \int_\Omega \beta_{\lambda}(u_0^{\delta,\lambda})(u_{0}^{\delta,\lambda}-u_0)
  +\delta\int_\Omega\beta_{\lambda}'(u_0^{\delta,\lambda})
  |\nabla u_0^{\delta,\lambda}|^2
  +\delta\int_\Gamma\beta_{\lambda}(u_{0,\Gamma}^{\delta,\lambda})
  \beta_{\Gamma,\lambda}(u_{0,\Gamma}^{\delta,\lambda})
  =0,
  \]
  which yields by monotonicity and the subdifferential relation for 
  $\widehat\beta_{\lambda}$ that
  \[
  \int_\Omega \widehat\beta_{\lambda}(u_0^{\delta,\lambda})  
  +\delta\int_\Gamma\beta_{\lambda}(u_{0,\Gamma}^{\delta,\lambda})
  \beta_{\Gamma,\lambda}(u_{0,\Gamma}^{\delta,\lambda})
  \leq\int_\Omega \widehat\beta_{\lambda}(u_0). 
  \]
  Hence, exploiting \eqref{ip_extra} on the Yosida approximations as
  \[
  \int_\Gamma\beta_{\lambda}(u_{0,\Gamma}^{\delta,\lambda})
  \beta_{\Gamma,\lambda}(u_{0,\Gamma}^{\delta,\lambda})
  \geq\frac1{2M}\int_\Gamma|\beta_{\Gamma,\lambda}
  (u_{0,\Gamma}^{\delta,\lambda})|^2 - C,
  \]
  we infer that 
  \[
  \int_\Omega \widehat\beta_{\lambda}(u_0^{\delta,\lambda})  
  +\frac\delta{2M}\int_\Gamma|\beta_{\Gamma,\lambda}
  (u_{0,\Gamma}^{\delta,\lambda})|^2
  \leq C + \int_\Omega \widehat\beta_{\lambda}(u_0).
  \]
  Consequently, taking the limit as $\lambda\searrow0$ and 
  using assumption \eqref{ip1_delta0} it is possible to prove that
    \[
  \int_\Omega \widehat\beta(u_0^{\delta})  
  +\frac\delta{2M}\int_\Gamma|\xi_{0,\Gamma}^\delta|^2
  \leq C + \int_\Omega \widehat\beta(u_0),
  \]
  which by comparison in \eqref{cf2} implies in particular that 
  \beq
  \label{cf4}
  \big\|\widehat\beta(u_0^\delta)\big\|_{L^1(\Omega)}+
  \delta\|\partial_\nu u_0^\delta\|^2_{H_\Gamma}\leq C.
  \eeq
  Now, collecting the information given by \eqref{cf3} and \eqref{cf4},
  using the elliptic regularity theory \cite[Thm.~3.2, p.~1.79]{BG87} and
  the trace theorems \cite[Thm~2.27, p.~1.64]{BG87} we infer that 
  \beq\label{cf5}
  \delta\|u_{0,\Gamma}^\delta\|^2_{V_\Gamma}\leq C.
  \eeq
  Then, the estimates \eqref{cf3}, \eqref{cf4}, and \eqref{cf5}
  allow us to conclude the proof.
\end{proof}

%%%%%%%%%%%%%%%%%%%%%%%%%%%%%%%%%%%%%%%%%%%%%%%%%%%%%%%%%%%%%%%%%%%%%%%%

\section*{Acknowledgments}
This research received a support from the Italian Ministry of Education, 
University and Research~(MIUR): Dipartimenti di Eccellenza Program (2018--2022) 
-- Dept.~of Mathematics ``F.~Casorati'', University of Pavia. 
TF acknowledges the support from the JSPS KAKENHI 
Grant-in-Aid for Scientific Research(C), Japan, Grant Number 21K03309 and 
from the Grant Program of The Sumitomo Foundation, Grant Number 190367.
PC and LS gratefully mention their affiliation
to the GNAMPA (Gruppo Nazionale per l'Analisi Matematica, 
la Probabilit\`a e le loro Applicazioni) of INdAM (Isti\-tuto 
Nazionale di Alta Matematica). Moreover, PC aims to point out his collaboration,
as Research Associate, to the IMATI -- C.N.R. Pavia, Italy.
LS was partially funded by the Austrian Science Fund (FWF)
through the Lise Meitner grant M 2876.

\end{document}